\documentclass[11pt]{article}

\usepackage{fullpage,amsmath,amssymb,color}
\usepackage{algorithm}
\usepackage{enumerate}
\usepackage{algpseudocode}
\newcommand{\BlackBox}{\rule{1.5ex}{1.5ex}}
\newenvironment{proof}{\par\noindent{\bf Proof\ }}{\hfill\BlackBox\\[2mm]}

\usepackage{graphicx}
\usepackage{epstopdf}
\usepackage{subfigure}
\graphicspath{{figures/}}
\usepackage{graphicx}   
\usepackage{multirow}
\usepackage{amsmath,authblk}
\usepackage{amsfonts}
\usepackage{latexsym}
\usepackage{verbatim}
\newtheorem{remark}{Remark}

\usepackage{amsmath}    
\usepackage[table]{xcolor}

\usepackage{multirow}
\newtheorem{assumption}{Assumption}
\usepackage [pagewise,displaymath, mathlines]{lineno}
\usepackage{amssymb}
\usepackage{epstopdf}
\usepackage{amsmath}
\usepackage{epsf}
\usepackage{graphicx}
\usepackage{mathdots}
\DeclareMathOperator*{\argmin}{argmin}

\newtheorem{definition}{Definition}[section]
\newtheorem{lemma}{Lemma}[section]

\newtheorem{theorem}{Theorem}[section]

\newcommand{\beq}{\begin{equation}}
\newcommand{\eeq}{\end{equation}}
\newcommand{\beqa}{\begin{eqnarray}}
\newcommand{\eeqa}{\end{eqnarray}}
\newcommand{\beqas}{\begin{eqnarray*}}
\newcommand{\eeqas}{\end{eqnarray*}}
\newcommand{\bi}{\begin{itemize}}
\newcommand{\ei}{\end{itemize}}
\newcommand{\ba}{\begin{array}}
\newcommand{\ea}{\end{array}}

\newcommand{\nn}{\nonumber}

\def\eqnok#1{(\ref{#1})}
\def\argmin{{\rm argmin}}

\def\vgap{\vspace*{.1in}}

\def\E{{\bf E}}
\def\exp{{\rm exp}}

\newcommand{\bbr}{\Bbb{R}}

\def\cS{{\cal S}}

\newcommand{\dd}{\xi}  
\newcommand{\pp}{x} 
\newcommand{\setpp}{\mathcal{X}} 

\title {Zeroth-order Nonconvex Stochastic Optimization: Handling Constraints, High-Dimensionality and Saddle-Points\footnote{Both authors contributed equally and are listed in alphabetical order}}
\author[1]{Krishnakumar Balasubramanian\thanks{kbala@ucdavis.edu}}
\author[2]{Saeed Ghadimi\thanks{sghadimi@princeton.edu}}
\affil[1]{Department of Statistics, University of California, Davis}
\affil[2]{Department of Operations Research and Financial Engineering, Princeton University}
\date{}
\begin{document}
\maketitle
\begin{abstract}
 In this paper, we propose and analyze zeroth-order stochastic approximation algorithms for nonconvex and convex optimization, with a focus on addressing constrained optimization, high-dimensional setting and saddle-point avoiding. To handle constrained optimization, we first propose generalizations of the conditional gradient algorithm achieving rates similar to the standard stochastic gradient algorithm using only zeroth-order information. To facilitate zeroth-order optimization in high-dimensions, we explore the advantages of  structural sparsity assumptions. Specifically, (i) we highlight an implicit regularization phenomenon where the standard stochastic gradient algorithm with zeroth-order information adapts to the sparsity of the problem at hand by just varying the step-size and (ii) propose a truncated stochastic gradient algorithm with zeroth-order information, whose rate of convergence depends only poly-logarithmically on the dimensionality. We next focus on avoiding saddle-points in non-convex setting. Towards that, we interpret the Gaussian smoothing technique for estimating gradient based on zeroth-order information as an instantiation of first-order Stein's identity. Based on this, we provide a novel linear-(in dimension) time estimator of the Hessian matrix of a function using only zeroth-order information, which is based on second-order Stein's identity. We then provide an algorithm for avoiding saddle-points, which is based on a zeroth-order cubic regularization Newton's method and discuss its convergence rates.
\end{abstract}

\newpage
\setcounter{equation}{0}
\section{Introduction}
In this work, we propose and analyze algorithms for solving the following stochastic optimization problem
\beq \label{eq:main_prob}
\min_{\pp \in \setpp} \left\{ f(\pp)  = \E_\dd[F(\pp, \dd)]= \int F(\pp, \dd) \, dP(\dd) \right\},
\eeq
where $\setpp$ is a closed convex subset of $\mathbb{R}^d$. The case of nonconvex objective function $f$ is ubiquitous in modern statistical machine learning problems and developing provable algorithms for such problems has been a topic of intense research in the recent years~\cite{nocedal2006nonlinear,  bertsekas1999nonlinear,jain2017non}, along with the more standard convex case~\cite{ben2001lectures, boyd2004convex,nesterov2013introductory,bertsekas2015convex, beck2017first}. Several methods are available for solving such stochastic optimization problems under access to different oracle information, for example, function queries (zeroth-order oracle), gradient queries (first-order oracle), Hessian-queries (second-order oracle), and similar higher-order oracles. In this work, we assume that we only have access to noisy evaluation of $f$ through a stochastic zeroth-order oracle described in detail in Assumption~\ref{unbiased_assum}. This oracle setting is motivated by several applications where only noisy function queries of problem \eqnok{eq:main_prob} is available and obtaining higher-order information might not be possible. Such a situation occurs frequently for example, in simulation based modeling~\cite{rubinstein2016simulation}, selecting the tuning parameters of deep neural networks~\cite{snoek2012practical} and design of black-box attacks to deep networks~\cite{chen2017zoo}. It is worth noting that recently such zeroth-order optimization techniques have also been applied in the field of reinforcement learning~\cite{salimans2017evolution, choromanski18a, mania2018simple}. Furthermore, methods using similar oracles have been studied in the literature under the name of derivative-free optimization~\cite{spall2005introduction,conn2009introduction}, bayesian optimization~\cite{mockus2012bayesian} and optimization with bandit feedback~\cite{bubeck2012regret}.


Algorithms available for solving problem \eqnok{eq:main_prob} depend crucially on the constraint set $\setpp$, along with the structure imposed on the objective function, $f$. Despite decades of work in zeroth-order optimization literature, there still exists several challenges, primarily motivated by contemporary statistical machine learning problems. A majority of the existing zeroth-order algorithms are predominantly analyzed in the low-dimensional unconstrained setting. Furthermore, when $f$ is non-convex, apart from the first-order stationarity result for gradient descent (GD) algorithm in~\cite{GhaLan12}, other meaningful theoretical results are lacking in the zeroth-order optimization literature. In this work, we provide theoretically sound algorithms to address the following three main drawbacks of existing zeroth-order optimization methods.

The first issue we address is that of \textbf{constrainted zeroth-order stochastic optimization}. For the problem in~\eqnok{eq:main_prob}, depending on the geometry of the constraint set $\setpp$, the cost of computing the projection to the set might be prohibitive. In the first-order oracle setting, this lead to the re-emergence of Conditional Gradient (CG) algorithms recently~\cite{hazan2012projection, jaggi2013revisiting}. But the performance of the CG algorithm under the zeroth-order oracle is unexplored in the literature to the best of our knowledge, both under convex and nonconvex settings. Hence it is natural to ask if CG algorithms, with access to zeroth-order oracle has similar convergence rates compared to zeroth-order GD algorithms for the unconstrained case. To address this question, we propose and analyze in Section~\ref{sec:vanillacgd} a classical version of CG algorithm with zeroth-order information and provide convergence results. We then propose a modification in Section~\ref{sec:improvedcgd} that has improved rates, when $f$ is convex. Notably, we demonstrate that with zeroth-order information, the complexity of CG algorithms also depend linearly on the dimensionality, similar to the GD algorithms, thereby facilitating constrained zeroth-order optimization.


Next, we consider the \textbf{impact of dimensionality in zeroth-order optimization}. Considering the unconstrained case of $\setpp = \mathbb{R}^d$, recall that when first-order information is available, the rate of convergence of the standard Gradient Descent (GD) algorithm is dimension-independent~\cite{nesterov2013introductory}. Whereas when only the zeroth-order information is available, any algorithm (with estimated gradients) has (at least) linear dependence on $d$~\cite{GhaLan12,jamieson2012query,duchi2015optimal}. This illustrates yet another difference between the availability of first and zeroth-order oracle information. We refer to this situation as the low-dimensional setting in the rest of the paper. This motivates us to examine assumptions under which one can achieve weaker dependence on the dimensionality while optimizing with zeroth-order~information. In a recent work ~\cite{wang18e}, the authors used a \emph{functional sparsity} assumption, under which the function $f:\mathbb{R}^d \to \mathbb{R}$ to be optimized depends only on $s$ of the $d$ components, and proposed a LASSO based algorithm that has poly-logarithmic dependence on the dimensionality when $f$ is convex. We refer to this situation as the high-dimensional setting. In this work, we perform a refined analysis under a similar sparsity assumption for both convex and nonconvex objective functions. When the performance is measured by the size of the gradient, we show in Section~\ref{sec:gdinhd} that zeroth-order GD algorithm (without using thresholding or LASSO approach of~\cite{wang18e}), has poly-logarithmic dependence on the dimensionality thereby demonstrating an \emph{implicit regularization} phenomenon in this setting. Note that this is applicable for both convex and nonconvex objectives. When the performance is measured by function values (as in the case of convex objective), we show that a simple thresholded zeroth-order GD algorithm achieves a poly-logarithmic dependence on dimensionality. This algorithm is notably less expensive than the algorithm proposed by~\cite{wang18e}.

Finally, we address the issue of \textbf{avoiding saddle-points in zeroth-order stochastic optimization}. When the function $f$ is non-convex, designing algorithms that avoid saddle-points and converge to local minimizers is challenging, as exemplified by worst-case computational hardness results~\cite{murty1987some, cartis2018second}. Hence, it is necessary to impose  further structure on the problem to obtain meaningful results. A particularly interesting structure on $f$ is the so-called strict saddle property, which necessitates that all local minima are global minima. This structure has regained popularity as several useful stochastic optimization problems in statistical machine learning are shown to posses this property; for example, phase retrieval~\cite{sun2018geometric}, tensor decomposition~\cite{ge2015escaping}, matrix completion and sensing~\cite{bhojanapalli2016global,ge2016matrix} and training deep neural networks~\cite{kawaguchi2019elimination}. See also the survey article~\cite{sun2015nonconvex}. Motivated by this, algorithms that avoid saddle-points and converge to second-order stationary points have re-gained popularity as well. Indeed, methods based on exact or in-exact second-order oracle naturally converge to second-order stationary points~\cite{NestPoly06-1, cartis2011adaptivea, cartis2011adaptiveb, xu2017newton, tripuraneni2017stochastic, carmon2018accelerated, allen2018natasha}. Furthermore, first-order methods escape saddle points by leveraging an additional noise term in each iteration; for example~\cite{ge2015escaping, jin2017escape, reddi2018generic} and the references therein. But to the best of our knowledge, there is no algorithm for efficiently avoiding saddle-points under zeroth-order oracle information. In this work, we propose a zeroth-order cubic regularized Newton method, that converges efficiently to second-order stationary points with just noisy function evaluations. In order to do so, we interpret the Gaussian smoothing for zeroth-order gradient estimation~\cite{NesSpo17}, as an instantiation of Stein's identity~\cite{stein1972bound, stein1981estimation}. Based on this interpretation, we develop provable techniques for estimating the Hessian of a function at a point with just function queries, leveraging higher-order Stein's identity. Notably, our Hessian estimator is based only on inner-product  evaluations thereby having a linear-in-dimension time runtime. We also provide a comprehensive complexity analysis of the proposed algorithm in terms of achieving second-order stationary points. \\

\noindent \textbf{Our contributions:} To summarize the above discussion, in this paper we make the following contributions to the literature on zeroth-order stochastic optimization.
\begin{enumerate}
\item We first analyze a classical version of CG algorithm in the nonconvex (and convex) setting, under access to zeroth-order information and provide results on the convergence rates in the low-dimensional setting. We then propose and analyze a modified CG algorithm in the convex setting with zeroth-order information and show that it attains improved rates in the low-dimensional setting.
\item Next, we consider a zeroth-order stochastic gradient algorithm in the high-dimensional nonconvex setting and illustrate an implicit regularization phenomenon --the algorithm converges to first-order stationary points with rates that depend only poly-logarithmically on dimensionality. We also propose a truncated zeroth-order stochastic gradient algorithm in the convex setting which also depends only poly-logarithmically on the dimensionally but has improved dependence on the error-tolerance. 
\item Finally, we propose a zeroth-order Stochastic cubic regularized Newton method that avoids saddle points and converges to second-order stationary points efficiently. Our algorithm is based on a novel technique for estimating the Hessian of a function from function queries based on Stein's identities. 
\end{enumerate}
Our contributions extend the applicability of zeroth-order stochastic optimization to the constrained, high-dimensional and non-convex settings and also provide theoretical insights in the form of rates of convergence. A summary of the results is provided in Table~\ref{tab:summary}. 

\begin{table}[t!]
\centering
\begin{tabular}{| c| c | c | c | c|}
\hline
 Algorithm  & Structure & Function Queries  & References   \\
 \hline
 \hline
 \multirow{2}{*}{ZSCG (Alg~\ref{alg_ZCGD})}  & Nonconvex & ${\cal O}(d/\epsilon^4)$   & \multirow{2}{*}{Theorem~\ref{theorem_CGD}} \\
\cline{2-3}
  & Convex &${\cal O}(d/\epsilon^3)$ &  \\
\hline
Modified ZSCG (Alg~\ref{alg_ZCGDSC2})  & Convex &${\cal O}(d/\epsilon^2)$ &  Theorem~\ref{theom_ZCGDSC_cvx} \\
\hline
ZSGD (Alg~\ref{alg_ZGD}) & Nonconvex, $s$-sparse &   ${\cal O}\left((s \log d)^2/\epsilon^4\right)$ &  Theorem~\ref{nocvx} \\
\hline
Truncated ZSGD (Alg~\ref{alg_TZGD})  & Convex, $s$-sparse & ${\cal O}\left(s( \log d/\epsilon)^2\right)$    & Theorem~\ref{thm:trunconv}   \\
\hline
\multirow{2}{*}{ZSGD}  & Convex &  ${\cal O}( d/\epsilon^2)$ &  \cite{jamieson2012query,duchi2015optimal, GhaLan12} \\
\cline{2-4}
 & Nonconvex &  ${\cal O}( d/\epsilon^4)$ & \cite{GhaLan12} \\
\hline
ZCRN (Alg~\ref{alg_ZSCRN}) & Nonconvex & ${\cal O} \left(\frac{d}{\epsilon^{3.5}}\right)+ \tilde {\cal O} \left(\frac{d^4}{\epsilon^{2.5}}\right) $& Theorem~\ref{thm:main_cubic}\\
\hline
\end{tabular}
\caption{A list of complexity bounds for stochastic zeroth-order methods to find an $\epsilon$-optimal or $\epsilon$-stationaly or $\epsilon$-local optimal (see Definition~\ref{def_complx}) point of problem \eqnok{eq:main_prob}. Here, $\tilde {\cal O}$ hides $\log$ factors in $d$.}
\label{tab:summary}
\end{table}

\subsection{Preliminaries}\label{sec:prelim}
We now list the main assumptions we make in this work. Additional assumptions will be introduced in the appropriate sections as needed. We start with the assumption on the zeroth-order oracle.
\begin{assumption}\label{unbiased_assum}
Let $\|\cdot  \|$ be a norm on $\mathbb{R}^d$. For any $\pp \in \bbr^d$, the zeroth-order oracle outputs an estimator $F(\pp,\dd)$ of $f(\pp)$ such that $\E[F(\pp,\dd)] = f(\pp), \E[\nabla F(\pp,\dd)] = \nabla f(\pp), \E[\|\nabla F(\pp,\dd) - \nabla f(\pp)\|_*^2] \le \sigma^2$, where $\|\cdot\|_*$ denotes the dual norm.
\end{assumption}
It should be noted that in the above assumption, we do not observe $\nabla F(\pp,\dd)$ and we just assume that it is an unbiased estimator of gradient of $f$ and its variance is bounded. Furthermore, we make the following smoothing assumption about the noisy estimation of $f$.
\begin{assumption}\label{smooth_assum}
Function $F$ has Lipschitz continuous gradient with constant $L$, almost surely for any $\dd$, i.e., $
\|\nabla F(y,\dd) -\nabla F(\pp,\dd)\|_* \le L\|y-\pp\|,
$
which consequently implies that\\
$
|F(y,\dd) - F(\pp,\dd) - \langle \nabla F(\pp,\dd), y-\pp \rangle| \le \frac{L}{2}\|y-\pp\|^2.
$
\end{assumption}
It is easy to see that the above two assumptions imply that $f$ also has Lipschitz continuous gradient with constant $L$ since 
\beq\label{fgrad_smooth}
\|\nabla f(y) - \nabla f(\pp)\|_* \le \E \left[\|\nabla F(y,\dd) - \nabla F(\pp,\dd)\|_*\right] \le L \|y-\pp\|
\eeq
due the Jensen's inequality for the dual norm. We now collect some facts about a gradient estimator based on the above zeroth-order~information. Let $ u \sim N(0,I_d)$ be a standard Gaussian random vector. For some $\nu \in (0,\infty)$ consider the smoothed function $f_\nu(\pp) = \E_u \left[f(\pp+ \nu u) \right]$. Nesterov~\cite{NesSpo17} has shown that $\nabla f_\nu(\pp) =$
\begin{align}\label{eq:gradest}
 \E_u \left[\frac{f(\pp+\nu u)}{\nu}~u\right] = \E_u \left[ \frac{f(\pp+\nu u) - f(\pp)}{\nu}~u\right]  = \frac{1}{(2\pi)^{d/2}} \int \frac{f(\pp+\nu u) - f(\pp)}{\nu}~u ~e^{-\frac{\|u\|_2^2}{2}}~du.
\end{align}
This relation implies that we can estimate gradient of $f_\nu$ by only using evaluations of $f$. In particular, one can define stochastic gradient of $f_\nu(\pp)$ as
\beq \label{def_gradk}
 G_\nu(\pp, \dd, u) = \frac{F(\pp+ \nu u, \dd) - F(\pp, \dd)}{\nu}~u,
\eeq
which is an unbiased estimator of $\nabla f_\nu(\pp)$ under Assumption~\ref{unbiased_assum} since
$$
\E_{u,\dd}[G_\nu(\pp, \dd, u)]= \E_{u}[\tfrac{f(\pp+ \nu u) - f(\pp)}{\nu}~u]=\nabla f_\nu(\pp).
$$
We leverage the following properties of $f_\nu$ due to Nesterov~\cite{NesSpo17} in our proofs later, which we replicate below for completeness.
\begin{theorem}[~\cite{NesSpo17}]\label{smth_approx} For a Gaussian random vector $u\sim N(0,I_d) $ we have that 
\begin{align}\label{eq:l2gauss}
\E[\|u\|^k] \le (d+k)^{k/2}
\end{align} 
for any $k \ge 2$. Moreover, the following statements hold for any function $f$ whose gradient is Lipschitz continuous with constant $L$.
\begin{itemize}
\item [a)] The gradient of $f_{\nu}$ is Lipschitz continuous with constant $L_{\nu}$ such that $L_{\nu} \le L$.
\item [b)] For any $x \in \bbr^d$,
\beqa
|f_{\nu}(x)-f(x)| &\le& \frac{\nu^2}{2} L d, \label{rand_smth_close}\\ 
\|\nabla f_{\nu}(x) - \nabla f(x)\| &\le& \frac{\nu}{2}L (d+3)^{\frac{3}{2}}\label{rand_smth_close_grad}.
\eeqa
\item [c)]
For any $x \in \bbr^n$,
\beq \label{stoch_smth_approx_grad}
\frac{1}{\nu^2}\E_u[\{f(x+\nu u)-f(x)\}^2\|u\|^2] \le \frac{ \nu^2}{2}L^2(d+6)^3 + 2(d+4)\|\nabla f(x)\|^2.
\eeq
\end{itemize}

\end{theorem}

We next introduce the Stein's identity, popular in the statistics and probability theory literature.
\begin{theorem}[~\cite{stein1972bound,stein1981estimation}]
Let $u \sim N(0,I_d)$, be a standard Gaussian random vector and let $g: \mathbb{R}^d \to \mathbb{R}$, be an almost-differentiable function~\footnote{For a definition of almost-differentiable function, we refer the reader to Definition 1 in~\cite{stein1981estimation}} with $\E\left[\|\nabla g(u)\| \right] < \infty$, we have
\begin{align}\label{eq:steinmain}
\E\left[u~g(u)\right] = \E\left[\nabla g(u)\right].
\end{align}
Furthermore, when the function $g$ has a twice continuously differentiable Hessian, $\nabla^2g(\cdot)$, we have the following (where the Expectation is assumed to exist):
\begin{align}\label{eq:secondstein}
\E[(u u^\top-I_d)~g(u)] = \E[\nabla^2g(u)].
\end{align}
\end{theorem}
Based on the above theorem, the Gaussian smoothing approach of estimating gradients from function queries proposed by~\cite{NesSpo17}, is indeed based on  Stein's identity. Indeed, if we let $g(u) = f(x+ \nu u)$ in  Equation~\ref{eq:steinmain}, it is easy to see that the identity in Equation~\ref{eq:gradest} holds by simply evaluating the Gaussian Stein's identity in Equation~\ref{eq:steinmain}. We elaborate more on this connection and extensions in Section~\ref{sec:second}. We conclude the section, by defining the following criterion which are used to analyze the complexity of our proposed algorithms.
\begin{definition} \label{def_complx}
Assume that a solution $\bar \pp \in \setpp$ as output of an algorithm and a target accuracy $\epsilon>0$ are given. Then:
\begin{itemize}
\item If $f$ is convex, $\bar \pp$ is called an $\epsilon$-optimal point of problem \eqnok{eq:main_prob} if $\E[f(\bar x)]-f(x_*) \le \epsilon$, where $x_*$ denotes an optimal solution of the problem.
\item If $f$ is nonconvex, $\bar \pp$ is called an $\epsilon$-stationary point of the unconstrained variant of problem \eqnok{eq:main_prob} if $\E[\|\nabla f(\bar x)\|_*] \le \epsilon$. For the constrained case, $\bar \pp$ should satisfies $\E[\langle \nabla f(\bar \pp), \bar \pp - u \rangle] \le \epsilon$ for all $u \in \setpp$.
\item If $f$ is nonconvex, $\bar \pp$ is called an $\epsilon$-local optima of the unconstrained variant of problem \eqnok{eq:main_prob} if 
\beq
\max \left\{\sqrt{\E[\|\nabla f(\bar{x})\|_*]}
,\frac{-1}{\sqrt{\lambda_{\max}(\nabla^2 f)}}~\E[\lambda_{\min} \left(\nabla^2 f(\bar{x}) \right)] \right\} \le \sqrt{\epsilon} \nonumber
\eeq
where for a symmetric matrix $A$, $\lambda_{\min}(A)$ and $\lambda_{\max}(A)$ denotes the minimum and maximum eigenvalue.
\end{itemize}
\end{definition}
It should be pointed out that while the above performance measures are presented in expectation form, one can also use their high probability counterparts. Since, convergence results in this case can be obtained by making sub-Gaussian tail assumptions on the output of the zeroth-order oracle and using the standard two-stage process presented in \cite{GhaLan12,lan2016conditional}, we do not elaborate more on this approach. Furthermore, note that the aforementioned measures for evaluating the algorithms are from the derivative-free optimization point of view. In the literature on optimization with bandit feedback, the preferred performance measure is the so-called regret of the algorithm~\cite{bubeck2012regret,shamir2013complexity} which may have a different behavior than our performance measures. 

\setcounter{equation}{0}
\section{Handling Constraints: Zeroth-order Stochastic Conditional Gradient Type Method}\label{sec:vanillacgd}
In this section, we study zeroth-order stochastic conditional gradient (ZSCG) algorithms in the low-dimensional setting for solving constrained stochastic optimization problems. In particular, we incorporate a variant of the gradient estimate defined in \eqnok{def_gradk} into the framework of the classical CG method and provide its convergence analysis in Subsection~\ref{sec:zscg}.  We also present improved rates for a variant of this method in Subsection~\ref{sec:improvedcgd} when $f$ is convex. Throughout this section, we assume that $\bbr^d$ is equipped with the self-dual Euclidean norm i.e., $\|\cdot\| = \|\cdot\|_2$. We also make the following natural boundedness assumption.

\begin{assumption}\label{bnd_grad}
The feasible set $\setpp$ is bounded such that $\max_{x,y \in \setpp} \|y-x\| \le D_\setpp$ for some $D_\setpp >0$. Moreover, for all $\pp \in \setpp$, there exists a constant $B>0$ such that $\| \nabla f(\pp)\| \leq B $.
\end{assumption}
We should point out that under Assumptions~\ref{unbiased_assum} and \ref{smooth_assum}, the second statement in Assumption~\ref{bnd_grad} follows immediately by the first one and choosing $B:= LD_\setpp +\|\nabla f(x_*)\|$. However, we just use $B$ in our analysis for simplicity.

\subsection{Zeroth-order Stochastic Conditional Gradient Method}\label{sec:zscg}\vspace{-0.07in}

The vanilla ZSCG method is formally presented in Algorithm~\ref{alg_ZCGD} and a few remarks about it follows.
\begin{algorithm} [t]
	\caption{Zeroth-order Stochastic Conditional Gradient Method}
	\label{alg_ZCGD}
	\begin{algorithmic}
\State Input: $z_0 \in \setpp$, smoothing parameter $\nu>0$, non-negative sequence $\alpha_k$, positive integer sequence $m_k$, iteration limit $N\geq 1$ and probability distribution $P_R(\cdot)$ over $\{1,\ldots,N\}$.

\For {$k = 1, \ldots, N$}
\State 1. Generate $u_k=[u_{k,1},\ldots, u_{k,m_k}]$, where $u_{k,j} \sim N(0,I_d)$, call the stochastic oracle to compute $m_k$ stochastic gradient $G_\nu^{k,j}$ according to \eqnok{def_gradk} and take their average:
\beq \label{def_bargradk}
\bar{G}_\nu^k \equiv \bar{G}_\nu(z_{k-1}, \dd_k, u_k) =  \frac{1}{m_k} \sum_{j=1}^{m_k} \frac{F(z_{k-1}+ \nu u_{k,j}, \dd_{k,j}) - F(z_{k-1}, \dd_{k,j})}{\nu}~u_{k,j}.
\eeq
\State 2. Compute
\begin{align}
\pp_{k}&= \underset{u \in \setpp}{\argmin} \langle \bar{G}_\nu^k , u \rangle, \label{def_xk_GD} \\
z_k & = (1- \alpha_k) z_{k-1} + \alpha_k \pp_k.\label{def_zk_GD}
\end{align}
\EndFor
\State Output: Generate $R$ according to $P_R(\cdot)$ and output $z_R$.
\end{algorithmic}
\end{algorithm}
First, note that this algorithm differs from the classical CG method in estimating the gradient using zeroth-order information and in outputting a random solution from the generated trajectory. This randomization scheme is the current practice in the literature to provide convergence results for nonconvex stochastic optimization (see e.g.,~\cite{GhaLan12,ReSrPoSm16}). Second, $\bar{G}_\nu^k$ is the averaged variant of the gradient estimator presented in Subsection~\ref{sec:prelim} and is still an unbiased estimator of $\nabla f_\nu(z_{k-1})$. Moreover, it can be easily seen that it has a reduced variance with respect to the individual estimators i.e.,
\beq \label{var0}
\E [\|\bar{G}_\nu^k - \nabla f_\nu (z_{k-1})\|^2] \le \frac{1}{m_k} \E [\|G_\nu^{k,j}- \nabla f_\nu(z_{k-1})\|^2].
\eeq
We emphasize that the use of the above variance reduction technique in stochastic CG methods is standard and has been previously proposed and leveraged in several works (see e.g.,~\cite{lan2016conditional,hazan2016variance,ReSrPoSm16, mokhtari2018conditional, mokhtari2018stochastic, Ghadimi2018}). Indeed, when exact gradient is not available, an error term appears in the convergence analysis which should converge to $0$ at a certain rate as the algorithm moves forward. Hence, the choice of $m_k$ plays a key role in the convergence analysis of Algorithm~\ref{alg_ZCGD}. $\bar{G}_\nu^k$  can be also viewed as a biased estimator for $\nabla f(z_{k-1})$. 
Finally, since $f$ is possibly nonconvex, we need a different criteria than the optimality gap to provide convergence analysis of Algorithm~\ref{alg_ZCGD}. The well-known Frank-Wolfe Gap given by
\beq
g_{_\setpp}^k \equiv g_{_\setpp}(z_{k-1}):=\langle \nabla f(z_{k-1}), z_{k-1} - \hat{\pp}_k \rangle, \ \ \text{where} \ \
\hat \pp_{k}= \underset{u \in \setpp}{\argmin} \langle  \nabla f(z_{k-1}) , u \rangle, \label{FWGap}
 \eeq
has been widely use in the literature to show rate of convergence of the CG methods when $f$ is convex (see e.g., \cite{FW56,DemRub70,Hearn82}). In this case, it is easy to see that
\beq \label{FWGap_cvx}
f(z_{k-1}) - f^* \leq g_{_\setpp}(z_{k-1}).
\eeq
When $f$ is nonconvex, this criteria is still useful since $
\langle \nabla f(z_{k-1}), z_{k-1}-u \rangle \le  g_{_\setpp}(z_{k-1}),~\forall u \in \setpp$, which implies that one can obtain an approximate stationary point of problem \eqnok{eq:main_prob} by minimizing $g_{_\setpp}^k $, in the view of Definition~\ref{def_complx}. Note that in our setting, this quantity is not exactly computable and it is only used to provide convergence analysis of Algorithm~\ref{alg_ZCGD} as shown in the next result.

\begin{theorem} \label{theorem_CGD}
Let $\{z_k\}_{k \ge 0}$ be generated by Algorithm~\ref{alg_ZCGD} and Assumptions~\ref{unbiased_assum}, \ref{smooth_assum}, and \ref{bnd_grad} hold.
\begin{enumerate}
\item Let $f$ be nonconvex, bounded from below by $f^*$, and let the parameters of the algorithm be set as
\beq \label{def_alpha_m}
\nu = \sqrt{\frac{2B_{L\sigma}}{ N (d+3)^3}}, \ \ \alpha_k = \frac{1}{\sqrt{N}}, \ \ m_k = 2B_{L \sigma} (d+5)N,  \ \ \forall k \ge 1
\eeq
for some constant $B_{L \sigma} \ge \max\{\sqrt{B^2+\sigma^2}/L,1\}$ and a given iteration bound $N \ge 1$. Then we have
\beq \label{CGD_nocvx}
\E[g_{_\setpp}^R] \le  \frac{f(z_0) -f^*+ L D_\setpp^2 + 2 \sqrt{B^2+\sigma^2}}{\sqrt{N}},
\eeq
where $R$ is uniformly distributed over $\{1,\ldots,N\}$ and $g_k$ is defined in \eqnok{FWGap}. Hence, the total number of calls to the zeroth-order stochastic oracle and linear subproblems required to be solved to find an $\epsilon$-stationary point of problem \eqnok{eq:main_prob} are, respectively, bounded by
\beq \label{CGD_nocvx2}
{\cal O}\left(\frac{d}{\epsilon^4}\right), \ \ {\cal O}\left(\frac{1}{\epsilon^2}\right).
\eeq

\item Let $f$ be convex and let the parameters be set to
\beq \label{def_alpha_m_cvx}
\nu = \sqrt{\frac{2B_{L\sigma}}{ N^2 (d+3)^3}}, \ \ \alpha_k = \frac{6}{k+5}, \ \ m_k =2B_{L \sigma} (d+5)N^2,  \ \ \forall k \ge 1.
\eeq
Then we have
\beq \label{CGD_cvx}
\E[f(z_N)] - f^* + \E[g_{_\setpp}^R] \leq \frac{120[f(z_0) - f(x_*)]}{(N+3)^3} + \frac{36L D^2_\setpp}{N+5} +
\frac{ \sqrt{B^2+\sigma^2} }{N}
\eeq
where $R$ is random variable from $\{1,\ldots,N\}$ whose probability distribution is given by
\beq \label{def_Gamma}
P_R (R=k) = \frac{\alpha_k \Gamma_N}{2 \Gamma_N (1-\Gamma_N)},
\quad\qquad
\Gamma_k = \prod_{i=1}^k\left( 1 - \frac{\alpha_i}{2} \right), \ \ \Gamma_0 =1.
\eeq
Hence, the total number of calls to the zeroth-order stochastic oracle and linear subproblems required to be solved to find and $\epsilon$-optimal solution of problem \eqnok{eq:main_prob} are, respectively, bounded by
\beq \label{CGD_cvx2}
{\cal O}\left(\frac{d}{\epsilon^3}\right), \ \ {\cal O}\left(\frac{1}{\epsilon}\right).
\eeq

\end{enumerate}
\end{theorem}

In order to prove Theorem~\ref{theorem_CGD}, we need the following result that provides upper bounds for the variance of our gradient estimator.
\begin{lemma}\label{bargrad_var}
Let $\bar{G}_\nu^k$ be computed by \eqnok{def_bargradk}. Then under Assumptions~\ref{unbiased_assum}, \ref{smooth_assum} and \ref{bnd_grad}, we have
\beqa
\E [\|\bar{G}_\nu^k - \nabla f_\nu (z_{k-1})\|^2] &\le& \frac{2(d+5) (B^2+\sigma^2)}{m_k}+ \frac{\nu^2}{2m_k} L^2 (d+3)^3, \label{var1}\\
\E [\|\bar{G}_\nu^k - \nabla f (z_{k-1})\|^2] &\le& \frac{4(d+5) (B^2+\sigma^2)}{m_k}  + \frac{3\nu^2}{2} L^2 (d+3)^3.\label{var2}
\eeqa
\end{lemma}

\begin{proof}[Proof of Lemma~\ref{bargrad_var}]
First note that using \eqnok{stoch_smth_approx_grad} for function $F$ instead of $f$, under Assumptions~\ref{unbiased_assum} and \ref{smooth_assum}, we obtain
\begin{align*}
\E [\|G_\nu^{k,j}\|^2] &\leq   \tfrac{\nu^2 L^2}{2} (d+6)^3+ 2\left[\| \nabla f(z_{k-1}) \|^2+\sigma^2\right] (d+4)
\end{align*}
Also noting \eqnok{def_gradk}, \eqnok{var0}, and the fact that $\|\nabla f_\nu\| \le B$ under Assumption~\ref{bnd_grad}, we have
\[
\E [\|\bar{G}_\nu^k - \nabla f_\nu (z_{k-1})\|^2] \le \frac{1}{m_k} \left(\E [\|G_\nu^{k,j}\|^2] + B^2\right),
\]
which together with the above relation clearly imply \eqnok{var1}.
We can then obtain \eqnok{var2} by noting \eqnok{rand_smth_close_grad} and the fact that
\[
\E [\|\bar{G}_\nu^k - \nabla f (z_{k-1})\|^2] \le
2\E [\|\bar{G}_\nu^k - \nabla f_\nu (z_{k-1})\|^2] + 2\E [\|\nabla f_\nu (z_{k-1}) - \nabla f(z_{k-1})\|^2].
\]
\end{proof}

\begin{proof}[Proof of Theorem~\ref{theorem_CGD}]
Denoting $\Delta_{k} = \bar{G}_\nu^k - \nabla f(z_{k-1})$, noting \eqnok{fgrad_smooth}, \eqnok{def_zk_GD}, and \eqnok{FWGap}, we have
\begin{align}
f(z_k) & \leq f(z_{k-1}) + \langle \nabla f (z_{k-1}) , z_k - z_{k-1}\rangle + \frac{L}{2} \| z_k - z_{k-1}\|^2 \nn \\
& =f(z_{k-1}) + \alpha_k \langle \nabla f (z_{k-1}) , \pp_k - z_{k-1}\rangle + \frac{L\alpha_k^2}{2} \| \pp_k - z_{k-1}\|^2 \nn \\
& \leq f(z_{k-1})+\alpha_k \langle \nabla f (z_{k-1}) , \hat{\pp}_k - z_{k-1}\rangle + \frac{L \alpha_k^2}{2}\left[ \| \pp_k - z_{k-1}\|^2 + \| \pp_k - \hat{\pp}_k \|^2\right] + \frac{\| \Delta_k\|^2}{2L} \nn \\
& \le f(z_{k-1}) - \alpha_k g_{_\setpp}^k + L D_\setpp^2 \alpha_k^2+ \frac{\| \Delta_k\|^2}{2L}, \label{CGD_proof1}
\end{align}
where the last inequality follows from boundedness of the feasible set, \eqnok{FWGap}, and the fact that
\[
\langle \nabla f(z_{k-1}) + \Delta_k, \pp_k - u \rangle \leq  0\ \ \forall u \in \setpp
\]
due to the optimality condition of \eqnok{def_xk_GD}. Taking expectation from both sides of the above inequality, summing them up, rearranging the terms, and noting Lemma~\ref{bargrad_var}, we obtain
\[
\sum_{k=1}^N \alpha_k \E[g_{_\setpp}^k] \le f(z_0) - f^* + L D_\setpp^2 \sum_{k=1}^N \alpha_k^2+ \frac{\nu^2}{2} LN (d+3)^3 + \frac{2(d+5) (B^2+\sigma^2)}{L}  \sum_{k=1}^N \frac{1}{m_k}.
\]
Hence, choosing $\alpha_k =\alpha_1$ and $m_k = m_1$ for all $k \ge 1$, and noting that $R$ is a uniform random variable, we have
\begin{align*}
\E[g_{_\setpp}^R] = \frac{\sum_{k=1}^N \E[g_{_\setpp}^k]}{N} = \frac{\sum_{k=1}^N \alpha_k \E[g_{_\setpp}^k]}{\sum_{k=1}^N \alpha_k}
&\leq \frac{f(z_0) -f^*}{N \alpha_1} + L D_\setpp^2 \alpha_1 + \frac{\nu^2}{2\alpha_1} L (d+3)^3 \nn \\
&+ \frac{2(d+5) (B^2+\sigma^2)}{L \alpha_1 m_1},
\end{align*}
which together with \eqnok{def_alpha_m} imply \eqnok{CGD_nocvx}. Hence, \eqnok{CGD_nocvx2} follows by noting that the total number of calls to the stochastic oracle is bounded by $\sum_{k=1}^N m_k$.

Now assume that $f$ is convex. Hence, by \eqnok{FWGap_cvx} and \eqnok{CGD_proof1}, we have
\[
f(z_k)- f(x_*) \le ( 1 - \tfrac{\alpha_k}{2})(f(z_{k-1} - f(x_*)) - \frac{\alpha_k g_{_\setpp}^k}{2} + L D_\setpp^2 \alpha_k^2+ \frac{\| \Delta_k\|^2}{2L}
\]

Taking expectation from both sides of the above inequality, dividing them by $T_k$, and summing them up, and noting \eqnok{def_Gamma}, we obtain
\begin{align*}
\frac{\E[f(z_N)] - f^*}{\Gamma_N} + \sum_{k=1}^N \frac{\alpha_k \E[g_{_\setpp}^k]}{2\Gamma_k}  \leq f(z_0) - f^* + L D^2_\setpp \sum_{k=1}^N \frac{\alpha_k^2}{\Gamma_k} + \frac{1}{2L} \sum_{k=1}^N \frac{\E \left[ \| \Delta_k \|^2\right]}{\Gamma_k},
\end{align*}
which together with the fact that
\[
\sum_{k=1}^N \frac{\alpha_k}{2\Gamma_k} = \frac{1-\Gamma_N}{\Gamma_N}, \ \ 1- \Gamma_1 \le 1- \Gamma_N \le 1
\]
due to \eqnok{def_Gamma}, imply that
\begin{align}
\E[f(z_N)] - f^* + \E[g_{_\setpp}^R]&\leq \frac{\Gamma_N}{ 1- \Gamma_N} \left[f(z_0) - f^* + L D^2_\setpp \sum_{k=1}^N \frac{\alpha_k^2}{\Gamma_k} +
\frac{2(d+5) (B^2+\sigma^2)}{L} \sum_{k=1}^N \frac{1}{\Gamma_k m_k} \right. \nn\\
& \qquad \qquad \qquad + \left. \frac{\nu^2}{2} L (d+3)^3 \sum_{k=1}^N \frac{1}{\Gamma_k}\right]
\end{align}
Now noting \eqnok{def_alpha_m_cvx} and \eqnok{def_Gamma}, we have
\begin{align}
&\Gamma_k = \frac{60}{(k+3)(k+4)(k+5)}, \qquad \qquad \sum_{k=1}^N \frac{\alpha_k^2}{\Gamma_k} \le \sum_{k=1}^N \frac{3(k+3)}{5} = \frac{3N(N+7)}{10}, \nn\\
&\Gamma_N \sum_{k=1}^N \frac{1}{\Gamma_k m_k} \le \frac{1}{4(d+5) B_{L \sigma} N}, \qquad \qquad \Gamma_N \sum_{k=1}^N \frac{1}{\Gamma_k} \le N.\nn
\end{align}
Combining the above relations, we get \eqnok{CGD_cvx} and \eqnok{CGD_cvx2}.
\end{proof}

\begin{remark}
Observe that the complexity bounds in \eqnok{CGD_nocvx2}, in terms of $\epsilon$, match the ones obtained in \cite{Ghadimi2018, ReSrPoSm16, mokhtari2018stochastic} for stochastic CG method with first-order oracle applied to  nonconvex problems. For convex problems, similar observation can be made for terms in \eqnok{CGD_cvx2} which match the ones in \cite{hazan2016variance, Ghadimi2018}. Note that the linear dependence of our complexity bounds on $d$ is unimprovable due to the lower bounds for zeorth-order algorithms applied to convex optimization problems \cite{duchi2015optimal}. We conjecture that this is also the case for nonconvex problems.
\end{remark}

\subsection{Improved Rates for Convex Problems}\label{sec:improvedcgd}
Our goal in this subsection  is to improve the complexity bounds of the ZCSG method when $f$ is convex. Recall that the ZSCG method presented in Section~\ref{sec:zscg} involves two main steps: the gradient evaluation step and the linear optimization step. Motivated by~\cite{lan2016conditional}, we now propose a modified algorithm that allows one to skip the gradient evaluation from time to time. Notice that, as our gradients are estimated by calling the zeroth-order oracle, this directly reduces the number of calls to the zeroth-order oracle. We first state a subroutine in Algorithm~\ref{alg_ZCGDsubroutine} used in our modified algorithm.
\begin{algorithm} [t]
	\caption{Inexact Conditional Gradient (ICG) method}
	\label{alg_ZCGDsubroutine}
	\begin{algorithmic}

\State Input: $(x, g, \gamma, \mu)$.
\State Set $\bar{y}_0 =x$, $t=1$, and $\kappa=0$..
\While {$\kappa =0$} 
\beq \label{def_gradk2}
y_t = \underset{u \in \setpp}{\argmin} \left\{  h_{\gamma}(u) :=\langle g + \gamma(\bar{y}_{t-1}- \pp), u - \bar{y}_{t-1} \rangle \right \}
\eeq
\State  If $h_{\gamma}(y_t) \geq -\mu$, set $\kappa=1$.
\State Else $\bar{y}_t  = \tfrac{t-1}{t+1} \bar{y}_{t-1} + \frac{2}{t+1} y_t $ and $t=t+1$.
\EndWhile
\State Output $\bar{y}_t$.
\end{algorithmic}
\end{algorithm}
Note that Algorithm~\ref{alg_ZCGDsubroutine} is indeed the zeroth-order conditional gradient method for inexactly solving the following quadratic program
\beq \label{qd_subproblem}
P_{\setpp}(x,g,\gamma)= \underset{u \in \setpp}{\argmin} \left\{ \langle g,u \rangle + \frac{\gamma}{2} \| u -\pp\|^2 \right\},
\eeq
which is the standard subproblem of stochastic first-order methods applied to a minimization problem when $g$ is an unbiased stochastic gradient of the objective function at $x$. We now present Algorithm~\ref{alg_ZCGDSC2} which applies the CG method to inexactly solve subproblems of the stochastic accelerated gradient method. This way of using CG methods can significantly improve the total number of calls to the stochastic oracle. Our next result provides convergence analysis of this algorithm.


\begin{algorithm} [H]
	\caption{Zeroth-order Stochastic Accelerated Gradient Method with Inexact Updates}
	\label{alg_ZCGDSC2}
	\begin{algorithmic}

\State Input:$z_0=x_0 \in \setpp$, smoothing parameter $\nu>0$, sequences $\alpha_k$, $m_k$, $\gamma_k$, $\mu_k$, and iteration limit $N\geq 1$.
\For {$k = 1, \ldots, N$ }
\State 1. Set \beq \label{def_w}
w_k = (1-\alpha_k) z_{k-1}+\alpha_k x_{k-1}
\eeq

\State 2. Generate $u_k=[u_{k,1},\ldots, u_{k,m_k}]$, where $u_{k,j} \sim N(0,I_d)$, call the stochastic oracle $m_k$ times to compute $\bar{G}_\nu^k \equiv \bar{G}_\nu(w_k, \dd_k, u_k)$ as given by \eqnok{def_bargradk}, and set 
\beq \label{CGS_update}
\pp_k = ICG(x_{k-1}, \bar{G}_\nu^k, \gamma_k, \mu_k),
\eeq
where $ICG(\cdot)$ is the output of Algorithm~\ref{alg_ZCGDsubroutine} with input $(x_{k-1}, \bar{G}_\nu^k, \gamma_k)$.
\State 3. Set
\beq \label{def_zk_av}
z_k  = (1- \alpha_k) z_{k-1} + \alpha_k \pp_k
\eeq
\EndFor
\State Output: $z_N$
\end{algorithmic}
\end{algorithm}

\begin{theorem}\label{theom_ZCGDSC_cvx}
Let $\{z_k\}_{k \geq 1}$ be generated by Algorithm~\ref{alg_ZCGDSC2}, the function $f$ be convex, and
\beqa \label{def_alpha_m_mu_cvx}
\alpha_k &=& \frac{2}{k+1}, \ \ \gamma_k = \frac{4L}{k}, \ \ \mu_k = \frac{L D_X^0}{k N}, \ \ \nu = \frac{1}{\sqrt{2N}} \max\left\{\frac{1}{d+3}, \sqrt{\frac{D_X^0}{d(N+1)}} \right\}\nn \\
m_k &=& \frac{k(k+1)}{D_X^0} \max\left\{(d+5)B_{L \sigma}N, d+3 \right\},  \ \ \forall k \ge 1,
\eeqa
and for some constants $D_X^0 \ge \|x_0-x_*\|^2$ and $B_{L \sigma} \ge \max\{\sqrt{B^2+\sigma^2}/L,1\}$. Then under Assumptions~\ref{unbiased_assum}, \ref{smooth_assum}, and \ref{bnd_grad}, we have
\beq \label{SCGD_cvx}
\E[f(z_N)-f(x_*)] \le  \frac{12 L D_X^0}{N (N+1)}.
\eeq
Hence, the total number of calls to the stochastic oracle and linear subproblems solved to find and $\epsilon$-stationary point of problem \eqnok{eq:main_prob} are, respectively, bounded by
\beq \label{SCGD_cvx2}
{\cal O}\left(\frac{d}{\epsilon^2}\right), \ \ {\cal O}\left(\frac{1}{\epsilon}\right).
\eeq
\end{theorem}

\begin{proof}
First, note that by \eqnok{fgrad_smooth}, we have
\beqa
f_\nu(z_k) &\leq& f_\nu(w_k)  + \langle \nabla f_\nu(w_k), z_k-w_k\rangle + \frac{L}{2} \|z_k - w_k \|^2 \nn \\
&\le& (1-\alpha_k) f_\nu(z_{k-1})  + \alpha_k \left[f_\nu(w_k)+\langle \nabla f_\nu(w_k), \pp_k-w_k\rangle \right] \nn \\ 
&+ & \frac{L \alpha_k^2}{2} \|\pp_k - \pp_{k-1} \|^2, \label{eq:cgdeq1}
\eeqa
where the second inequality follows from convexity of $f_\nu$, \eqnok{def_w}, and \eqnok{def_zk_av}. Also note that by \eqnok{def_gradk2} and \eqnok{CGS_update}, we have
\beq\label{eq:cgdeq2}
-\mu_k  \leq  \langle \bar{G}_\nu^k+ \gamma_k (\pp_k -\pp_{k-1}), u -\pp_k \rangle \qquad \forall u \in \setpp.
\eeq
Letting $u = \pp_*$ in the above inequality and multiplying it by $\alpha_k$, summing it up with \eqnok{eq:cgdeq1}, and denoting $\bar \Delta_{k} = \bar{G}_\nu^k - \nabla f_\nu(w_k)$, we obtain
\[
f_\nu(z_k) \le (1-\alpha_k) f_\nu(z_{k-1})  + \alpha_k f_\nu(x_*)+ \alpha_k \left[\mu_k+ \langle \bar \Delta_k+\gamma_k(\pp_k-\pp_{k-1}), \pp_*-\pp_k \rangle\right]+ \frac{L \alpha_k^2}{2} \|\pp_k - \pp_{k-1} \|^2,
\]
which together with the facts that
\begin{align*}
&\| \pp_{k-1} - \pp_* \|^2 = \| \pp_k - \pp_{k-1}\|^2 + \| \pp_k - \pp_*\|_2^2 + 2 \langle \pp_{k-1} - \pp_{k}, \pp_k - \pp_*\rangle, \\
&\alpha_k \langle \bar \Delta_k, \pp_*-\pp_k \rangle  \le \alpha_k \langle \bar \Delta_k, \pp_*-\pp_{k-1} \rangle + \frac{\|\bar \Delta_k\|^2}{2L}+\frac{L \alpha_k^2}{2}\|\pp_k - \pp_{k-1}\|^2,
\end{align*}
imply
\beqa
f_\nu(z_k) \le (1-\alpha_k) f_\nu(z_{k-1})  + \alpha_k f_\nu(x_*)&+& \alpha_k \left[\mu_k+\frac{2L\alpha_k-\gamma_k}{2}\|\pp_k - \pp_{k-1} \|^2+\langle \bar \Delta_k, \pp_*-\pp_{k-1} \rangle\right] \nn \\
&+& \frac{\alpha_k\gamma_k}{2}\left[\| \pp_{k-1} - \pp_* \|^2 - \| \pp_k - \pp_* \|^2\right]+\frac{\|\bar \Delta_k\|^2}{2L}.
\eeqa
Defining
\beq \label{def_Gamma2}
\hat \Gamma_k = \prod_{i=2}^k\left( 1 - \alpha_i \right), \ \ \hat \Gamma_1 =1,
\eeq
subtracting $f_\nu(x_*)$ from both sides of the above inequality, diving them by $\hat \Gamma_k$, taking expectation, summing them up, noting \eqnok{rand_smth_close}
assuming that $\alpha_1=1$, $\gamma_k \ge 2L \alpha_k$, and $\gamma_k \alpha_k/\hat \Gamma_k$ is constant for any $k \ge 1$, we obtain
\[
\frac{\E \left[ f_(z_N)\right] - f(x_*)-\nu^2 L d}{\hat \Gamma_N} \leq \frac{\gamma_1}{2}\| \pp_0 - \pp_* \|^2+\sum_{k=1}^N \frac{\alpha_k \mu_k}{\hat \Gamma_k} +\left[\frac{(d+5) (B^2+\sigma^2)}{L} +\frac{\nu^2 L(d+3)^3}{2}\right] \sum_{k=1}^N \frac{1}{m_k \hat \Gamma_k}.
\]
Now noticing that
\begin{align*}
&\hat \Gamma_k = \frac{2}{k(k+1)}, \qquad \frac{\alpha_k \gamma_k}{\hat \Gamma_k} = 4L, \qquad \frac{\alpha_k \mu_k}{\hat \Gamma_k} = \frac{L D_0^2}{N},\\
&\frac{1}{m_k \hat \Gamma_k} \le \frac{2 D_0^2}{\max\left\{(d+5)B_{L \sigma}N, d+3 \right\}}
\end{align*}
due to \eqnok{def_alpha_m_mu_cvx} and \eqnok{def_Gamma2}, we obtain \eqnok{SCGD_cvx}.

Furthermore, note that the function $h_\gamma$ defined in Algorithm~\ref{alg_ZCGDsubroutine} is indeed negative the FW-gap of the CG method applied to problem \eqnok{qd_subproblem}. From classical analysis of the CG method and similar to our result in Theorem~\ref{theorem_CGD}, one can show that the FW-gap is bounded by $L D_\setpp^2/T$ if the CG method runs for $T$ iteration. Since the gradient of the objective function in \eqnok{qd_subproblem} is Lipschitz continuous with constant $\gamma$, we have
\[
- h_{\gamma_k}(\bar y_{T_k}) \le \frac{\gamma_k D_\setpp^2}{T_k},
\]
which together with the choice of $\mu_k$ and $\gamma_k$ in \eqnok{def_alpha_m_mu_cvx}, imply that at iteration $k$ of Algorithm~\ref{alg_ZCGD}, we need to run Algorithm~\ref{alg_ZCGDsubroutine} for at most $T_k = 4 D_\setpp^2 N/D_0^2$ iterations. Therefore, the total number of iterations of Algorithm~\ref{alg_ZCGDsubroutine} to find an $\epsilon$-stationary point of problem \eqnok{eq:main_prob} is bounded by $\sum_{k=1}^N T_k \le 48L D_\setpp^2 /\epsilon^2$ due to \eqnok{SCGD_cvx2}.

\end{proof}

\begin{remark}
Observe that while the number of linear subproblems required to find an $\epsilon$-optimal solution of problem \eqnok{eq:main_prob} is the same for both Algorithms~\ref{alg_ZCGD}
 and \ref{alg_ZCGDSC2}, the number of calls to the stochastic zeroth-order oracle in Algorithm~\ref{alg_ZCGDSC2}
 is significantly smaller than that of Algorithm~\ref{alg_ZCGD}. It is also natural to ask if such an improvement is achievable when $f$ is nonconvex. This situation is more subtle and the answer depends on the performance measure used to measure the rate of convergence. Indeed, we can obtain improved complexity bounds for a different performance measure than the Frank-Wolfe gap with a modified algorithm. However, the complexity bounds are of the same order as \eqnok{CGD_nocvx2} in terms of the Frank-Wolfe gap for the modified algorithm. For the sake of completeness, we add this algorithm and its convergence analysis in in Section~\ref{sec:nonconveximproved}.
 \end{remark}

\subsection{Zeroth-order Stochastic Gradient Method with Inexact Updates-Nonconvex case}\label{sec:nonconveximproved}
In this section, we present a zeroth-order stochastic gradient method which applies the CG method to solve the subproblems. This algorithm shares the main idea of Algorithm~\ref{alg_ZCGDSC2}, but for nonconvex problems. We show while this algorithm enjoys better complexity bound than  Algorithm~\ref{alg_ZCGDSC2}, it possess the same one when the same performance measure is employed.
\begin{algorithm} [H]
	\caption{Zeroth-order Stochastic Gradient Method with Inexact Updates}
	\label{alg_ZCGDSC}
	\begin{algorithmic}

\State Input: $x_0 \in \setpp$, smoothing parameter $\nu>0$, positive integer sequence $m_k$, and sequences $\gamma_k$ and $\mu_k$ and a probability distribution $P_R(\cdot)$ over $\{0,\ldots,N-1\}$
\For {$k = 1, \ldots, N$} 
\State Generate $u_k=[u_{k,1},\ldots, u_{k,m_k}]$, where $u_{k,j} \sim N(0,I_d)$, call the stochastic oracle $m_k$ times, compute $\bar{G}_\nu^k \equiv \bar{G}_\nu(\pp_{k-1}, \dd_k, u_k)$ as given by \eqnok{def_bargradk}, and set $\pp_k$ to \eqnok{CGS_update}.
\EndFor
\State Output: Generate $R$ according to $P_R(\cdot)$ and output $x_R$.
\end{algorithmic}
\end{algorithm}
Since we are now using the CG method for inexactly solving \eqnok{qd_subproblem}, we can provide an alternative termination criterion than the FW-gap given in \eqnok{FWGap} to provide our convergence analysis. In particular, we use the gradient mapping defined as
\beq \label{grad_map}
GP_{\setpp}(x,g,\gamma) = \gamma (x-P_{\setpp}(x,g,\gamma)),
\eeq
where $P_{\setpp}$ is the solution to \eqnok{qd_subproblem}. This quantity which has been widely used in the literature as a convergence criteria for solving nonconvex problems (see, e.g., \cite{nemyud:83,Nest04}), plays an analogues role of the gradient in constrained problems. Next result provides some properties for this criteria.

\begin{lemma}\label{grad_gap}
Let $P_{\setpp}(\cdot)$ be defined in \eqnok{qd_subproblem}, $\gamma>0$, and $\pp \in \setpp$ are given.
\begin{itemize}
\item [a)] for and $\hat g \in \bbr^d$, we have
\[
\|P_{\setpp}(\pp,g,\gamma) -  P_{\setpp}(\pp,\hat g,\gamma)\| \leq \frac{\| g - \hat{g}\|}{\gamma}.
\]

\item [b)] Let $P^\mu_{\setpp}$ be the inexact solution of \eqnok{qd_subproblem} such that
\beq \label{inex_qp}
\langle g+ \gamma (P^\mu_{\setpp}(\pp,g,\gamma) - \pp),u - P^\mu_{\setpp}(\pp,g,\gamma) \rangle \ge -\mu \qquad \forall u \in \setpp
\eeq
for some $\mu \ge 0$. Then, we have
\[
\|P_{\setpp}(\pp,g,\gamma) -  P^\mu_{\setpp}(\pp,g,\gamma)\|^2 \leq \frac{\mu}{\gamma}.
\]

\item [c)] Let $g_{_\setpp}(\cdot)$ be the Frank-Wolfe gap defined in \eqnok{FWGap}. Then we have
\[
\|GP_{\setpp}(x,\nabla f(x),\gamma)\|^2 \le g_{_\setpp}(x).
\]
Moreover, under Assumption~\ref{bnd_grad}, we have
\[
g_{_\setpp}(x) \le (B/\gamma +D_\setpp)\|GP_{\setpp}(x,\nabla f(x),\gamma)\|.
\]
\end{itemize}

\end{lemma}

\begin{proof}
First note that \eqnok{qd_subproblem} implies
\[
\|P_{\setpp}(\pp,g,\gamma) -  P_{\setpp}(\pp,\hat g,\gamma)\| = \|\Pi_{\setpp}(\pp-g/\gamma) -  \Pi_{\setpp}(\pp-\hat g/\gamma)\| \le \frac{\| g - \hat{g}\|}{\gamma},
\]
where the last inequality follows from Lipschitz continuity of the Euclidian projection over the feasible set $\Pi_{\setpp}$. Second, by optimality condition of \eqnok{qd_subproblem}, we have
\beq\label{g_p1}
\langle g+ \gamma (P_{\setpp}(\pp,g,\gamma) - \pp),u - P_{\setpp}(\pp,g,\gamma) \rangle \ge 0 \qquad \forall \tilde u \in \setpp.
\eeq
Letting $\tilde u = P^\mu_{\setpp}(\pp,g,\gamma)$ in the above inequality and $u = P_{\setpp}(\pp,g,\gamma)$ and $g=\nabla f(x)$ in \eqnok{inex_qp} and summing them up, we clear get the result in part b). Third, letting $\tilde u=x$ in \eqnok{g_p1}, we have
\[
\|GP_{\setpp}(x,\nabla f(x),\gamma)\|^2 \le \gamma \langle \nabla f(x) ,x - P_{\setpp}(\pp,\nabla f(x),\gamma) \rangle \le \gamma g_{_\setpp}(x),
\]
where the last inequality follows from \eqnok{FWGap}. Furthermore, \eqnok{g_p1} also implies that
\begin{align*}
g_{_\setpp}(x) + \frac{1}{\gamma} \|GP_{\setpp}(x,\nabla f(x),\gamma)\|^2 &\le \langle \nabla f(x)+\gamma(x-u), x-P_{\setpp}(\pp,\nabla f(x),\gamma) \rangle \\ &\le (B/\gamma +D_\setpp)\|GP_{\setpp}(x,\nabla f(x),\gamma)\|,
\end{align*}
where the last inequality follows from Assumption~\ref{bnd_grad}.

%

\end{proof}

Now we are ready to state the main result for the nonconvex case.

\begin{theorem}\label{theom_ZCGDSC_nocvx}
Let $\{x_k\}$ be generated by Algorithm~\ref{alg_ZCGDSC}, the function $f$ be nonconvex, and
\beq \label{def_alpha_m_mu}
\nu = \sqrt{\frac{1}{ 2 N (d+3)^3}}, \ \ \gamma_k = 2L, \ \ \mu_k = \frac{1}{4N}, \ \ m_k = 6(d+5)N,  \ \ \forall k \ge 1.
\eeq
Then under Assumptions~\ref{unbiased_assum}, \ref{smooth_assum}, and \ref{bnd_grad}, we have
\beq \label{SCGD_nocvx}
\E[\|GP_{\setpp}(\pp_{R},\nabla f(\pp_{R}),\gamma_R)\|^2] \le  \frac{8L}{N} \left(f(x_0) -f^*+ L+B^2+\sigma^2 \right).
\eeq
where $R$ is uniformly distributed over $\{0,\ldots,N-1\}$ and $g_{\setpp}$ is defined in \eqnok{grad_map}. Hence, the total number of calls to the stochastic oracle and linear subproblems solved to find and $\epsilon$-stationary point of problem \eqnok{eq:main_prob} are, respectively, bounded by
\beq \label{SCGD_nocvx2}
{\cal O}\left(\frac{d}{\epsilon^2}\right), \ \ {\cal O}\left(\frac{1}{\epsilon^2}\right).
\eeq
\end{theorem}

\begin{proof}
First note that by \eqnok{fgrad_smooth}, we have
\[
f(\pp_k)  \leq f(\pp_{k-1}) + \langle \nabla f(\pp_{k-1}) , \pp_k - \pp_{k-1}\rangle + \frac{L}{2} \| \pp_k - \pp_{k-1}\|^2.
\]
Letting $u = \pp_{k-1}$ in \eqnok{eq:cgdeq2}, summing it up with the above inequality, and denoting $\Delta_{k} = \bar{G}_\nu^k - \nabla f(\pp_{k-1})$, we obtain
\begin{align*}
f(\pp_k) & \leq f(\pp_{k-1}) - \gamma_k \left(1 - \frac{L}{2\gamma_k} \right) \| \pp_k - \pp_{k-1}\|^2 +  \langle \Delta_k, \pp_{k-1} - \pp_k \rangle +\mu_k\\
&\leq  f(\pp_{k-1}) - \gamma_k \left(1 - \frac{L}{\gamma_k} \right)  \| \pp_k - \pp_{k-1}\|^2 + \frac{\|\Delta_k\|^2}{2L}+\mu_k.  \\
\end{align*}
Taking expectation from the above inequalities, summing them up, re-arranging the terms, and in the view of Lemma~\ref{bargrad_var}, we have
\begin{align*}
&~\sum_{k=1}^N \gamma_k \left(1 - \frac{L}{\gamma_k} \right) \E[\| \pp_k - \pp_{k-1}\|^2] \\ 
\leq&~f(\pp_0) - f^* + \sum_{k=1}^N \mu_k + \frac{\nu^2 L (d+3)^3 N}{2} + \frac{2(d+5) (B^2+\sigma^2)}{L} \sum_{k=1}^N \frac{1}{m_k},
\end{align*}
which together with the facts that $x_k = P^{\mu_k}_{\setpp}(\pp_{k-1},\bar{G}_\nu^k,\gamma_k)$ and
\begin{align*}
&~\frac{1}{\gamma_k^2}\|GP_{\setpp}(\pp_{k-1},\nabla f(\pp_{k-1}),\gamma_k)\|^2 \\ =&~\|\pp_{k-1}-P_{\setpp}(\pp_{k-1},\nabla f(\pp_{k-1}),\gamma_k)\|^2 \\
\le &~2 \|\pp_k - \pp_{k-1}\|^2+\frac{4 \mu_k}{\gamma_k}+\frac{4 \nu^2 L^2 (d+3)^3}{\gamma_k^2} + \frac{16(d+5) (B^2+\sigma^2)}{\gamma_k^2 m_k},
\end{align*}
imply that
\begin{align*}
&\sum_{k=1}^N \left(\frac{\gamma_k -L}{2\gamma_k^2}\right)\E[\|GP_{\setpp}(\pp_{k-1},\nabla f(\pp_{k-1}),\gamma_k)\|^2] \leq f(\pp_0) - f^* + \sum_{k=1}^N \left(\frac{3\gamma_k-2L}{\gamma_k}\right)\mu_k \\
&\qquad + \frac{\nu^2 L (d+3)^3}{2} \sum_{k=1}^N \left(1+\frac{4L(\gamma_k-L)}{\gamma_k^2}\right)+ \frac{2(d+5) (B^2+\sigma^2)}{L} \sum_{k=1}^N \frac{1}{m_k}\left(1+\frac{8L(\gamma_k -L)}{\gamma_k^2}\right).
\end{align*}
Hence, noting \eqnok{def_alpha_m_mu}, we obtain
\begin{align*}
&~\E[\|GP_{\setpp}(\pp_{R},\nabla f(\pp_{R}),\gamma_R)\|^2] \\ \leq &~\frac{8L[f(\pp_0) - f^*]}{N}+ 16 L^2 \mu_1+ 8 \nu^2 L^2 (d+3)^3  + \frac{48(d+5) (B^2+\sigma^2)}{ m_1 },
\end{align*}
which implies \eqnok{SCGD_nocvx}. Rest of the proof is similar to that of Theorem~\ref{theom_ZCGDSC_cvx} and hence we skip the details.
\end{proof}
\begin{remark}
We point out that while the complexity bounds in \eqnok{SCGD_nocvx2} are better than those in \eqnok{CGD_nocvx2} in terms of dependence on the target accuracy $\epsilon$, they have been obtained for a different performance measure. Indeed, if only the Frank-Wolfe gap is considered then it is easy to see that both bounds are of the same order of magnitude due to part c of Lemma~\ref{grad_gap}.
\end{remark}

\setcounter{equation}{0}
\section{Handling High-Dimensionality: Zeroth-order Stochastic Gradient Methods}\label{sec:gdinhd}
In this section, we study unconstrained variant of problem \ref{eq:main_prob} i.e, $\setpp = \mathbb{R}^d$, under certain sparsity assumptions on the objective function $f$ to facilitate zeroth-order optimization in high-dimensions. Recently,~\cite{wang18e} considered the convex case and proposed algorithms for high-dimensional zeroth-order stochastic optimization. Motivated by~\cite{wang18e}, we make the following assumption. \vspace{-0.1in}
\begin{assumption}\label{spars_assum}
For any $x \in \mathbb{R}^d$, we have $\| \nabla f(\pp) \|_0 \leq s$, i.e., the gradient is $s$-sparse, where $s \ll d$.
\end{assumption}
Note that the above assumption implies $\| \nabla f(\pp) \|_2 \leq \sqrt{s}\| \nabla f(\pp)\|_\infty$ and $\| \nabla f(\pp) \|_1 \leq s\| \nabla f(\pp)\|_\infty$, for all $\pp \in \mathbb{R}^d$. Furthermore, this assumption also implies that  $\| \nabla f_\nu(\pp) \|_0 \leq s$ for all $\pp \in \mathbb{R}^d$ since $\nabla f_\nu(\pp) = \E_u \left[\nabla f(\pp+ \nu u) \right]$. To exploit the above sparsity assumption, we assume that the primal space $\bbr^d$ is equipped with the $l_\infty$ norm throughout this section. More specifically, we assume that Assumptions~\ref{unbiased_assum} and \ref{smooth_assum} hold with the choice of $\|\cdot\|=\|\cdot\|_\infty$ and its dual norm $\|\cdot\|_*=\|\cdot\|_1$. We now present zeroth-order stochastic gradient methods for solving problem \eqnok{eq:main_prob} when $f$ is nonconvex and convex, in Subsections~\ref{sec:sparsenc} and~\ref{sec:sparsec} respectively.

\subsection{Zeroth-order Stochastic Gradient Method for Nonconvex Problems}\label{sec:sparsenc}\vspace{-0.07in}
In this subsection, we consider the zeroth-order stochastic gradient method presented in \cite{GhaLan12} (provided in Algorithm~\ref{alg_ZGD} for convenience) and provide a refined convergence analysis for it under the sparsity assumption~\ref{unbiased_assum}, when $f$ is nonconvex. Our main convergence result for Algorithm~\ref{alg_ZGD} under the gradient sparsity assumption is stated below.

\begin{algorithm} [t]
	\caption{Zeroth-Order Stochastic Gradient Method}
	\label{alg_ZGD}
	\begin{algorithmic}

\State Input: $x_0 \in \bbr^d$, smoothing parameter $\nu>0$, iteration limit $N\geq 1$, a probability distribution $P_R$ supported on $\{0,\ldots,N-1\}$.
\For {k =1, \ldots, N}
\State Generate $u_k \sim N(0,I_d)$, call the stochastic oracle, and compute $G_\nu(\pp_{k-1}, \dd_k, u_k)$ as defined in \eqnok{def_gradk} and set  $\pp_{k}=\pp_{k-1} - \gamma_k G_\nu(\pp_{k-1}, \dd_k; u_k)$.
\EndFor
\State Output: Generate $R$ according to $P_R(\cdot)$ and output $x_R$.
\end{algorithmic}
\end{algorithm}

\begin{theorem}\label{nocvx}
Let $\{\pp_k\}_{k \ge 0}$ be generated by Algorithm~\ref{alg_ZGD} and stepsizes are chosen such that $\forall k \ge 1$,
\beq
\gamma_k = \frac{1}{2 L \hat C \log d} \min \left\{\frac{1}{12 \hat s \log d}, \sqrt{\frac{ D_0 L \hat C}{2N \sigma^2}}\right\}, \quad
\nu \le \frac{1}{\sqrt{L \hat C \log d}} \min \left\{\sqrt{\frac{2 \sigma^2}{L}}, \sqrt{\frac{D_0}{N}} \right\} \label{def_nu}
\eeq
for some $\hat s \ge s$, $\hat C \ge C$ (the universal constant defined in Lemma~\ref{lem:infmoment}), and $D_0 \ge f(x_0)-f^*$.
Assume that $f$ is nonconvex. Then under Assumptions~\ref{unbiased_assum}, \ref{smooth_assum}, and \ref{spars_assum}, we have
\beq \label{main_nocvx}
\E_\zeta \left[ \| \nabla f(\pp_R) \|_1^2 \right] \leq \frac{150 L \hat C D_0 \hat s s (\log d)^2  }{N} +\frac{54 \sigma \sqrt{2L \hat C D_0}~ s \log d }{\sqrt{N}},
\eeq
where $\zeta = \{\dd, u, R\}$ and $R$ is uniformly distributed over $\{0,\ldots,N-1\}$. Hence, the total number of calls to the stochastic oracle (number of iterations) required to find an $\epsilon$-stationary point of problem \eqnok{eq:main_prob}, in the view of Definition~\ref{def_complx}, is bounded by
\beq\label{gd_bnd}
{\cal O}\left(\frac{(\hat s \log d)^2}{\epsilon^4}\right).
\eeq

\end{theorem}

Before proving the theorem, we first present two technical results which play key roles in our convergence analysis.
\begin{lemma} \label{lem:infmoment}
Let $ u \sim N(0,I_d) $ be a $d$-dimensional standard Gaussian vector.  Then for all integer $k \geq 1$ and for some universal constant $C$, we have $ \E \left[\| u\|^k_{\infty}  \right] \leq C (2  \log d)^{k/2} $.
\end{lemma}

\begin{proof}
Let $Z = \| u\|_{\infty}$ and denote by $p(x)$ the standard normal pdf. Note that we have
\begin{align*}
\E Z^k &= \int_{0}^{\infty} k x^{k-1} P(Z>x) \, dx \\
& \leq \int_0^{x_d}   k x^{k-1} dx+\int_{x_d}^\infty x^{k-2} p(x) dx
\end{align*}
where we define $x_d = \sqrt{2 \ln d}$.
Now we have
\begin{align*}
 \int_0^{x_d}   k x^{k-1} dx = x_d^k = (2 \log d)^{k/2}
\end{align*}
and by l'Hospital's rule, for large $d$ we have
\begin{align*}
 \int_{x_d}^\infty x^{k-2} p(x) dx \approx x_d^{k-3} p(x_d) \ll  (\log d)^{(k-3)/2} = o\left(\frac{(\log d)^{k/2}}{d}\right)
\end{align*}
Hence we have for some universal constant $C$,  $$ \E \left[\| u\|^k_{\infty}  \right] \leq C (2  \log d)^{k/2} .$$
\end{proof}

\begin{lemma}\label{approx_infty}
The following statements hold for function $f$ and its smooth approximation $f_\nu$. 
\begin{itemize}
\item [a)] Under Assumptions~\ref{unbiased_assum}  and \ref{smooth_assum}, gradient of $f$ is Lipschitz continuous with constant $L$ and $$
| f_\nu(\pp) - f(\pp) | \le \nu^2 C L \log d.$$
\item [b)] If Assumption~\ref{spars_assum} also holds, we have
\begin{align*}
\| \nabla f_\nu(\pp) - \nabla f(\pp) \|_2 &\le  C \nu L \sqrt{2s} (\log d)^{3/2}\\
\E \left[ \| G_\nu (\pp,\dd;u)\|^2_\infty\right] &\le 4C(\log d)^2 \left[L^2\nu^2 (\log d)+ 4\| \nabla f(\pp) \|_1^2+4\sigma^2\right].
\end{align*}
\end{itemize}
\end{lemma}


\begin{proof}
First note that
\begin{align*}
| f_\nu(\pp) - f(\pp) | & = \left| \E \left[f(\pp+\nu u) - f(\pp) - \nu \langle\nabla f(\pp), u  \rangle \right] \right| \\&\le
\E \left|f(\pp+\nu u) - f(\pp) - \nu \langle\nabla f(\pp), u  \rangle\right| \\
&\leq \frac{\nu^2 L}{2} \E \left[ \| u \|_\infty^2 \right]  \leq C\nu^2 L \log d, \\
\end{align*}
where the last inequality follows from Lemma~\ref{lem:infmoment}. Second, noting this lemma again, Assumption~\ref{spars_assum}, and part a), we have
\begin{align*}
\| \nabla f_\nu(\pp) - \nabla f(\pp) \|_2 &\le \sqrt{s^*}\| \nabla f_\nu(\pp) - \nabla f(\pp) \|_\infty \\
&\leq \frac{\sqrt{s}}{\nu (2\pi)^{d/2}} \int  | f(\pp+\nu u) - f(\pp) - \nu \langle\nabla f(\pp), u  \rangle|~ \| u\|_{\infty} e^{-\frac{\|u\|_2^2}{2}}~du \\
& \leq \frac{\nu L \sqrt{s}}{2(2\pi)^{d/2}} \int \| u \|^3_\infty e^{-\frac{\|u\|_2^2}{2}}~du \leq C \nu L \sqrt{2s} (\log d)^{3/2}.
\end{align*}
Furthermore, by \eqnok{def_gradk}, Holder inequality, Lemma~\ref{lem:infmoment}, and under Assumption~\ref{spars_assum} we have
\begin{align*}
&~\E \left[ \| G_\nu (\pp,\dd;u)\|^2_\infty\right] \\ = &~\frac{2}{\nu^2}\E \left[ |F(\pp+\nu u,\dd) - F(\pp,\dd) - \nu \langle\nabla F(\pp,\dd), u  \rangle|^2 \|u\|^2_\infty\right]+2\E \left[\langle\nabla F(\pp,\dd), u  \rangle^2 \| u\|_\infty^2 \right]  \\
\leq &~\tfrac{\nu^2 L^2}{2} \E \left[ \| u\|_\infty^6\right] + 2 \E_\dd[\| \nabla F(\pp,\dd) \|_1^2] \E_u\left[ \| u\|_\infty^4\right]\\
 \leq &~4CL^2\nu^2(\log d)^3+ 8C(\log d)^2 \E_\dd[\| \nabla F(\pp,\dd) \|_1^2] \\
 \leq &~4C(\log d)^2 \left[L^2\nu^2 (\log d)+ 4\| \nabla f(\pp) \|_1^2+4\sigma^2\right].
\end{align*}
\end{proof}

\begin{proof}[Proof of Theorem~\ref{nocvx}]
%
%
%
%
%
%
Noting \eqnok{def_gradk}, Lemma~\ref{approx_infty}.a), and with the notion of $G_{\nu,k} \equiv G_\nu (\pp_k,\dd_k,u_k)$, we have
\begin{align*}
f (\pp_{k+1}) &\le f(\pp_k) + \langle \nabla f(\pp_k), x_{k+1}-x_k \rangle +\frac{L}{2}\|x_{k+1}-x_k\|_\infty^2 \\ &\le f(\pp_k) - \gamma_k \langle \nabla f(\pp_k) , G_{\nu,k}\rangle + \frac{L \gamma_k^2}{2}  \| G_{\nu,k} \|_\infty^2,
\end{align*}
which after taking expectation imply that
\begin{align*}
\E[f (\pp_{k+1})]& \le f(\pp_k) - \gamma_k \|\nabla f(\pp_k)\|_2^2 + \gamma_k\langle \nabla f(\pp_k) , \nabla f(\pp_k)-\nabla f_\nu(\pp_k)\rangle + \frac{L \gamma_k^2}{2}  \E[\| G_{\nu,k} \|_\infty^2] \\
&\le f(\pp_k) - \frac{\gamma_k}{2} \|\nabla f(\pp_k)\|_2^2 + \frac{\gamma_k}{2} \|\nabla f(\pp_k)-\nabla f_\nu(\pp_k)\|_2^2 + \frac{L \gamma_k^2}{2}  \E[\| G_{\nu,k} \|_\infty^2] \\
&\le f(\pp_k) - \frac{\gamma_k}{2s} \left(1 -16 L C s(\log d)^2\gamma_k \right) \|\nabla f(\pp_k)\|_1^2 + (\nu L C)^2 s (\log d)^3 \gamma_k  \\
&\qquad + 2 L C (\log d)^2  \left[L^2\nu^2 (\log d)+ 4\sigma^2\right]\gamma_k^2,
\end{align*}
where the last inequality follow from Holder inequality and Lemma~\ref{approx_infty}.b). Summing both sides of the above inequality over the iterations and rearranging terms, we get
\[
\E[\|\nabla f(\pp_R)\|_1^2] \le \frac{6 s \left[f(x_0)-f^*+ (\nu L C)^2 s (\log d)^3 \sum_{k=1}^N\gamma_k+ 2CL (\log d)^2 \left(L^2\nu^2 (\log d)+ 4\sigma^2\right)\sum_{k=0}^{N-1} \gamma_k^2\right]}{\sum_{k=0}^{N-1} \gamma_k},
\]
where $R$ is uniformly distributed over $\{0,\ldots,N-1\}$ since
\[
\E[ \|\nabla f(\pp_R)\|_1^2]  = \frac{1}{N}\sum_{k=0}^{N-1} \|\nabla f(\pp_k)\|_1^2 =\frac{\sum_{k=0}^{N-1} \gamma_k \left(1 -16 L C s(\log d)^2\gamma_k \right) \|\nabla f(\pp_k)\|_1^2 }{\sum_{k=0}^{N-1} \gamma_k \left(1 -16 L C s (\log d)^2\gamma_k \right)}, 
\]
due to the constant choice of $\gamma_k$ in \eqnok{def_nu}. Therefore, we have
\[
\E[\|\nabla f(\pp_R)\|_1^2] \le 6 s \left[\frac{f(x_0)-f^*}{N\gamma_1}+ (\nu L C)^2 s (\log d)^3 + 2CL (\log d)^2 \left(L^2\nu^2 (\log d)+ 4\sigma^2\right) \gamma_1\right],
\]
which together with the choice of smoothing parameter in \eqnok{def_nu} imply \eqnok{main_nocvx}.
\end{proof}

\begin{remark}
Note that the above theorem establishes rate of convergence of Algorithm~\ref{alg_ZGD} which only poly-logarithmically depends on the problem dimension $d$, by just selecting the step-size appropriately, under additional assumption that the gradient is sparse. This significantly improves the linear dimensionality dependence of the rate of convergence of this algorithm as presented in \cite{GhaLan12} for general nonconvex smooth problems. 
\end{remark}
\begin{remark}
Remarkably, Algorithm~\ref{alg_ZGD} does not require any special operation to adapt to the sparsity assumption. This demonstrates an \emph{implicit regularization} phenomenon exhibited by the zeroth-order stochastic gradient method in the high-dimensional setting when the performance is measured by the size of the gradient in the dual norm. We emphasize that the choice of the performance measure is motivated by the fact that we allow $f$ to be nonconvex. Trivially, the result also applies to the case when $f$ is convex, for the same performance measure. 
\end{remark}
\subsection{Zeroth-order Stochastic Gradient Method for Convex Problems}\label{sec:sparsec}
We now consider the case when the function $f$ is convex. In this setting, a more natural performance measure is the convergence of optimality gap in terms of the function values. For this situation, we propose and analyze a truncate variant of Algorithm~\ref{alg_ZGD} that demonstrates similar poly-logarithmic dependence on the dimensionality. To proceed, in addition to Assumption~\ref{spars_assum}, we also make the following sparsity assumption on the optimal solution of problem \eqnok{eq:main_prob}.
\begin{assumption}\label{spars_assum1}
Problem \eqnok{eq:main_prob} has a sparse optimal solution $x_*$ such that $\|x_*\|_0 \le s^*$, where $s^* \approx s$.
\end{assumption}
Our algorithm for the convex setting is presented in Algorithm~\ref{alg_TZGD}. Note that this algorithm could be considered as a truncated variant of Algorithm~\ref{alg_ZGD} and a zeroth-order stochastic variant of the truncated gradient descent algorithm~\cite{jain2014iterative}. In the next result, we present convergence analysis of this algorithm.

\begin{algorithm} [t]
	\caption{Truncated Zeroth-Order Stochastic Gradient Method}
	\label{alg_TZGD}
	\begin{algorithmic}

\State Given a positive integer $\hat s$, replace updating step of Algorithm~\ref{alg_ZGD} with
\beq \label{def_trunc_xk}
\pp_{k}=  P_{\hat s} \left(\pp_{k-1} - \gamma_k G_\nu(\pp_{k-1}, \dd_k; u_k)\right),
\eeq
where $P_{\hat s}(\pp)$ keeps the top $\hat s$ largest absolute value of components of $\pp$ and make the others $0$.

\end{algorithmic}
\end{algorithm}

\begin{theorem}\label{thm:trunconv}
Let $\{\pp_k\}_{k \ge 1}$ be generated by Algorithm~\ref{alg_ZGD}, $f$ is convex, Assumptions~\ref{unbiased_assum}, \ref{smooth_assum},  \ref{spars_assum}, and~\ref{spars_assum1} hold. Also assume the stepsizes are chosen such that, $\forall k \ge 1$,
\beq
\gamma_k = \frac{1}{4 \hat C \hat s \log d} \min \left\{\frac{1}{12 L \hat s \log d}, \sqrt{\frac{ D_X^0 \hat C \hat s}{3N \sigma^2}}\right\}, \quad
\nu \le \sqrt{\log d} \min \left\{\frac{\sigma}{\log d}, \sqrt{\frac{\hat s^2 D_X^0}{N}} \right\} \label{def_gamma_cvx}
\eeq
for some $\hat C \ge C$, $\hat s \ge \max\{s,s^*\}$, and $D_X^0 \ge \|x_0-x_*\|^2$.
\beq \label{main_cvx}
\E \left[ f(\bar{x}_N) - f^*\right] \leq \frac{52 L \hat C D_X^0 \hat s^2 (\log d)^2  }{N} +\frac{69 \sigma \sqrt{3 \hat C D_X^0 \hat s} ~ \log d }{\sqrt{N}},
\eeq
where $\bar{x}_N = \frac{\sum_{k=0}^{N-1}x_k}{N}$.
Hence, the total number of calls to the stochastic oracle (number of iterations) required to find an $\epsilon$-optimal point of problem \eqnok{eq:main_prob} is bounded by 
\beq\label{gd_bnd_cvx}
{\cal O}\left(\hat s\left(\frac{\log d}{\epsilon}\right)^2\right).
\eeq
\end{theorem}
\begin{proof}
Denoting the index set of nonzero elements of $x_k$ and $x_*$ by $Z^k \subseteq \bbr^{\hat s}$ and $Z^* \subseteq \bbr^{s^*}$, respectively, and $J^k = Z^k \cup Z^{k+1} \cup Z^*$, we have
\begin{align*}
&~\| \pp_{k+1}  - \pp_* \|_2^2 \\
=&~\| \pp_{k+1}^{J^k}  - \pp_*^{J^k} \|_2^2 = \| \pp_{k}^{J^k}  - \pp_*^{J^k} - \gamma_k G_{\nu,k}^{J^k} \|_2^2=
\| \pp_{k}^{J^k}  - \pp_*^{J^k}\|_2^2+ \gamma_k^2 \|G_{\nu,k}^{J^k} \|_2^2 - 2\gamma_k \langle \pp_{k}^{J^k}  - \pp_*^{J^k}, \gamma_k G_{\nu,k}^{J^k} \rangle
\\
\le&~\| \pp_{k}  - \pp^* \|_2^2 + (2\hat s+s^*)\gamma_k^2 \|G_{\nu,k} \|_\infty^2  - 2 \gamma_k \langle \pp_{k}  - \pp_*, G_{\nu,k}  \rangle,
\end{align*}
where the inequality follows from the facts that $|J^k| \le 2\hat s+s^*$ and $\|G_{\nu,k}^{J^k} \| \le \|G_{\nu,k}\|$. Taking expectation from both sides of the above inequality, summing them up, noting Lemma~\ref{approx_infty}, convexity of $f_\nu$ (due to convexity of $f$), we have
\begin{align*}
\E \left[\| \pp_{N}  - \pp^* \|_2^2\right] & \leq \| \pp_{0}  - \pp^* \|_2^2 + (2\hat s+s^*) \sum_{k=0}^{N-1}  \gamma_k^2 \E \left[\| G_{\nu,k}\|_\infty^2 \right] - 2 \sum_{k =0}^{N-1}   \gamma_k \langle \pp_k - \pp^* ,\nabla f_\nu (\pp_k) \rangle \\
& \leq \| \pp_{0}  - \pp^* \|_2^2 +4C(2\hat s+s^*)(\log d)^2 \sum_{k=0}^{N-1}  \gamma_k^2 \left[L^2\nu^2 (\log d)+ 4\| \nabla f(\pp_k) \|_1^2+4\sigma^2\right] \\
&-2\sum_{k =0}^{N-1} \gamma_k \left[ f_\nu(\pp_k) - f_\nu(x_*)\right]\\
& \leq \| \pp_{0}  - \pp^* \|_2^2 +4C(2\hat s+s^*)(\log d)^2 \sum_{k=0}^{N-1}  \gamma_k^2 \left[L^2\nu^2 (\log d)+ 4\sigma^2\right] + 4\nu^2 C L \log d \sum_{k =0}^{N-1} \gamma_k \\
&-2\sum_{k =0}^{N-1} \gamma_k [1-16LCs(2\hat s+s^*)(\log d)^2 \gamma_k][f(\pp_k) - f(x_*)],
\end{align*}
where the last inequality follows from the fact that $f(x_k)-f(x_*) \ge 1/(2Ls)\|\nabla f(x_k)\|_2^2$ due to the convexity of $f$ and sparsity of its gradient.
Rearranging the terms in the above inequality and noting that $\bar{x}_N = \frac{\sum_{k=0}^{N-1}x_k}{N}$, we obtain
\[
f(\bar \pp_N) - f(x_*) \le \frac{\| \pp_{0}  - \pp^* \|_2^2 +4C(2\hat s+s^*)(\log d)^2 \sum_{k=0}^{N-1}  \gamma_k^2 \left[L^2\nu^2 (\log d)+ 4\sigma^2\right] + 4\nu^2 C L \log d \sum_{k =0}^{N-1} \gamma_k}{2\sum_{k =0}^{N-1} \gamma_k [1-16LCs(2\hat s+s^*)(\log d)^2 \gamma_k]}
\]
since
\[
\bar \pp_N = \frac{\sum_{k=0}^{N-1} x_k}{N} = \frac{\gamma_k [1-16LCs(2\hat s+s^*)(\log d)^2 \gamma_k] x_k}{\sum_{k=0}^{N-1} \gamma_k [1-16LCs(2\hat s+s^*)(\log d)^2 \gamma_k]}
\]
due to the constant choice of $\gamma_k$ in \eqnok{def_gamma_cvx}.  Hence,~\eqnok{main_cvx} follows by using the choice of parameters in \eqnok{def_gamma_cvx} into the above relation.
\end{proof}

\begin{remark}
While for convex case, similar to the nonconvex case, the complexity of Algorithm~\ref{alg_TZGD} depends poly-logarithmically on $d$, it only linearly depends on the choice of $\hat s$, facilitating zeroth-order stochastic optimization in high-dimensions under sparsity assumptions. 
\end{remark}
\begin{remark}
As discussed in detail in~\cite{wang18e}, both Assumption~\ref{spars_assum} and~\ref{spars_assum1} are implied when we assume the function $f$ depends on only $s$ of the $d$ coordinates. But, both Assumption~\ref{spars_assum} and~\ref{spars_assum1} are comparatively weaker than that assumption. Furthermore, unlike \cite{wang18e}, we do not make any assumption on the sparsity or smoothness of the second-order derivative of the objective function $f$ for our results. 
\end{remark}
\begin{remark}
As mentioned before,~\cite{wang18e} considers only the convex case. Furthermore, their gradient estimator with zeroth-order oracle requires $\text{poly}(s, s^*, \log d)$ function queries in each iteration whereas our estimator is based on only one function query per iteration. Moreover,~\cite{wang18e} requires computationally expensive debiased Lasso estimators whereas our method requires only simple thresholding operations (for convex case) to handle sparsity. 
\end{remark}

\setcounter{equation}{0}
\section{Handling Saddle-Points: Zeroth-Order Cubic Regularization Method}
In this section, we study zeroth-order stochastic cubic regularized Newton method for  unconstrained version of Problem~\ref{eq:main_prob}. Throughout this section, we equip our space with the self- dual Euclidean norm, i.e., $\| \cdot\| = \| \|_2$. Furthermore, for a matrix $A$, we denote by $\|A\|_F$, its Frobenious norm and by $\|A\|$, its operator norm. We also make the following smoothness assumption on the Hessian of the objective function $f$, which is a generalization of the assumption in Equation~\ref{fgrad_smooth}.
\begin{assumption}\label{a:fhess_smooth}
The function $f$ is twice differentiable and has Lipschitz continuous Hessian i.e., there exists $L_H>0$ such that
\[
\|\nabla^2 f(y)-\nabla^2 f(y)\| \le L_H \|y-x\| \ \ \forall x,y \in \bbr^d.
\]
\end{assumption}
It can be easily seen that the above assumption is equivalent to
\begin{align}
&\|\nabla f(y)-\nabla f(x)-\nabla^2 f(x)(y-x)\| \le \frac{L_H}{2} \|y-x\|^2,\label{fhess_smooth0} \\
&| f(y) - f(x) -  \langle \nabla f(x) , y-x \rangle  - \frac 12 \langle y-x, \nabla^2f(x) (y-x) \rangle | \leq \frac{L_H}{6} \|y-x\|^3.\label{fhess_smooth}
\end{align}
Note that such an assumption in standard in the analysis of second-order optimization techniques~\cite{NestPoly06-1}. We next describe a general technique for estimating the Hessian of a function based on Stein's identity in Section~\ref{sec:second} and use it to provide a zeroth-order cubic regularization method and its analysis in Section~\ref{sec:zocubic}.

\subsection{Estimating Hessian with Zeroth-Order Information}\label{sec:second}
Charles Stein, in his seminal paper~\cite{stein1972bound}, proposed a method for characterizing Gaussian random variables. Specifically, a random vector, $u \sim N(0,I_d)$, is standard Gaussian \emph{if and only if}, $\E\left[u~g(u)\right] = \E\left[\nabla g(u)\right]$, for all absolutely continuous function $g$. Note that Stein's identity, naturally relate function queries (left hand side of Equation~\ref{eq:steinmain}) to gradients (right hand side of Equation~\ref{eq:steinmain}) and thus is naturally suited for zeroth-order optimization. Indeed the Gaussian smoothing technique proposed by~\cite{NesSpo17}, is based on the Stein's identity. Indeed, if we let $g(u) = f(x+ \nu u)$ in  Equation~\ref{eq:steinmain}, it is easy to see that the identity in Equation~\ref{eq:gradest} holds by simply evaluating the Gaussian Stein's identity in Equation~\ref{eq:steinmain}.  Recall that the results in Sections~\ref{sec:vanillacgd} and~\ref{sec:gdinhd} are essentially based on approximately estimating the gradient information based on the Gaussian smoothing technique~\cite{NesSpo17}.  In this section, we develop techniques for approximately estimating the Hessian using zeroth-order oracle, based on second-order Stein's identities. It is worth noting that~\cite{erdogdu2016newton} also use Stein's identities to estimate the Hessian but they only work in the  restricted framework of generalized linear models with Gaussian data. Our use of Stein's identity to estimate Hessians, is completely different and we provide Hessian estimators for a general class of non-covnex, smooth functions, even for deterministic functions.

The second-order Gaussian Stein's identity, that we provide here informally for convenience, states thats $\E[(u u^\top-I_d)~g(u)] = \E[\nabla^2g(u)]$, for all functions $g$ with well-defined Hessians. Similar to first-order Stein's identity, this naturally relates function queries to Hessians. In order to leverage this, similar to the previous case, we let $g(u) = f(x+ \nu u)$ and note that we have
\begin{align}\label{eq:hessest}
\E \left[\frac{( uu^\top - I_d) f(x + \nu u)}{\nu^2} \right]= \E [\nabla^2f(x+ \nu u) ]  = \nabla^2 f_\nu (x) = H_{f_\nu}.
\end{align}
This provides a way of approximately estimating the Hessian of the function $f_\nu$ by approximating the expectation on the left hand side using Gaussian samples. Hence, we can leverage this estimate of Hessian of the smoothed function to get an approximate estimate of Hessian of $f$. Similar to the gradient-free setting, we now have the following estimates of the Hessian.
\begin{align}
H_{\nu} (\pp, \dd, u)&= \frac{1}{\nu^2} \left(u u^\top - I_d\right) F(x + \nu u, \dd),\nn \\
H_{\nu} (\pp, \dd, u)&= \frac{1}{\nu^2} \left(u u^\top - I_d\right) \left[F(x + \nu u, \dd)  - F(x,\dd) \right],\nn \\
H_{\nu} (\pp, \dd, u)&= \frac{1}{2\nu^2}\left(u u^\top - I_d\right) \left[F(x + \nu u, \dd)- F(x,\dd)  + F(x - \nu u, \dd)- F(x,\dd)\right],\label{Hessian_estimator} \\
\bar{H}_{\nu}& = \frac{1}{b}\sum_{i=1}^b H_{\nu} (\pp, \dd, u_i). \label{Hessian_estimator_redvar}
\end{align}
Note that above quantities are all unbiased estimators of $H_{f_\nu}$, with the last one, also being the variance reduced version. Note however that the above two estimators, unlike the gradient case, are not robust w.r.t the smoothing parameter $\nu$ in the sense that their variances blow up when $\nu$ converges to $0$. Hence, for the rest of this section we only focus on the Hessian estimator defined in \eqnok{Hessian_estimator}. The above Hessian-estimator has several advantages that we elaborate now. Recall that for second-order optimization algorithms, for example, cubic regularization, it is important to be able to compute the Hessian matrix operating on a vector $v \in \mathbb{R}^d$, efficiently. For the proposed estimator above, such a Hessian-vector product boils down to just inner-product based operations and could be done in time linear in the dimensionality. To the best of our knowledge, no such estimator for computing the hessian of a function exists. We now present the following results that characterize the estimation and approximation capability of the $H_\nu(x,\dd,u)$.


\begin{lemma}[Variance Bound]
Let the Hessian estimator be defined in \eqnok{Hessian_estimator} and Assumption~\ref{a:fhess_smooth} hold for $F(x,\dd)$. Then, we have
\beq\label{Hessian_var_bnd_fr}
\E[\| H_{\nu} (\pp, \dd, u)\|_F^4] \leq (d+16)^8 \left(\frac{L_H^4 (d+16)^2 \nu^4}{9}+ L^4 \right).
\eeq
As a consequence, we have
\beq\label{Hessian_var_bnd}
\E[\| H_{\nu} (\pp, \dd, u)\|^2] \leq  (d+16)^4 \left(\frac{L_H^2 (d+16)\nu^2}{3}+ L^2 \right).
\eeq
\end{lemma}

\begin{proof}
Noting \eqnok{Hessian_estimator} and Holder's inequality, we have
\begin{align*}
\E[\| H_{\nu} (\pp, \dd, u)\|_F^4] &= \E\left[\left\|\frac{1}{2\nu^2}\left(u u^\top - I_d\right) \left[F(x + \nu u, \dd)  + F(x - \nu u, \dd)- 2F(x,\dd)\right]\right\|_F^4\right] \\
&\le \frac{\Big(\E \left[|F(x + \nu u, \dd)  + F(x - \nu u, \dd)- 2F(x,\dd)|^8\right] \cdot \E\left[\|u u^\top - I_d\|_F^8\right]\Big)^\frac 12 }{16 \nu^8},
\end{align*}
which together with  Assumptions~\ref{smooth_assum}, assumption \eqnok{fhess_smooth} for $F(x,\dd)$, and the fact that
\begin{align*}
\E\left[\|u u^\top - I_d\|_F^8\right]= \E\left[\left(\|u\|^4-2\|u\|^2+d\right)^4\right] \le \E\left[\left(\|u\|^4+d\right)^4\right]
\le 2(d+16)^8,
\end{align*}
due to~\eqnok{eq:l2gauss}, imply
\begin{align*}
\E[\| H_{\nu} (\pp, \dd, u)\|_F^4] &\le \frac{(d+16)^4}{8\sqrt{2}} \left(\E \left[ \left|\frac{L_H \nu \|u\|^3}{3}+ \langle u, \nabla^2 F(x,\dd)u \rangle \right|^8 \right]\right)^\frac 12\\
&\le (d+16)^4 \left(\frac{L_H^4 \nu^4 \sqrt{\E \left[\|u\|^{24}\right]}}{81}+ L^4 \sqrt{\E \left[\|u\|^{16}\right]}  \right)\\
&\le (d+16)^8 \left(\frac{L_H^4 \nu^4 (d+16)^2}{9}+ L^4   \right).
\end{align*}
Moreover, by Holder's inequality, we have
\[
\E[\| H_{\nu} (\pp, \dd, u)\|^2] \le \E[\| H_{\nu} (\pp, \dd, u)\|_F^2] \le \left(\E[\| H_{\nu} (\pp, \dd, u)\|_F^4]\right)^\frac 12,
\]
which together with \eqnok{Hessian_var_bnd_fr}, imply \eqnok{Hessian_var_bnd}.
\end{proof}

\begin{lemma}[Approximation Error]
Under Assumption~\ref{a:fhess_smooth}, denoting the Hessian of $f$ by $H_f$ for simplicity, we have
\beq \label{hessian_approx_error}
\| H_{f_\nu} - H_f \| \leq \frac{L_H \nu (d+6)^{\frac{5}{2}}}{4}.
\eeq
\end{lemma}
\begin{proof}
Taking $y=x+\nu u$ in Equation~\ref{fhess_smooth}, note that we have
\beq
| f(x+\nu u) - f(x) - \nu \langle \nabla f(x) , u \rangle  - \frac{\nu^2}{2} \langle y-x, \nabla^2f(x) (y-x) \rangle | \leq \frac{L_H \nu^3}{6} \|u\|^3.\label{fhess_smooth2}
\eeq
Furthermore, note that
\begin{align*}
H_{f_\nu} - H_f &  = \frac{\E [( uu^\top - I_d)(f(x + \nu u)-f(x)+f(x - \nu u)-f(x))] } {2\nu^2} - H_f \\
& = \E \left[ \left(\frac{f(x + \nu u)-f(x)-\tfrac{\nu^2}{2} \langle u, H_f u \rangle+f(x - \nu u)-f(x)-\tfrac{\nu^2}{2} \langle u, H_f u \rangle}  {2\nu^2} \right)\left(uu^\top - I_d\right) \right],
\end{align*}
which together with \eqnok{fhess_smooth2} and~\eqnok{eq:l2gauss}, imply that
\[
\|H_{f_\nu} - H_f\| \le \frac{L_H \nu}{6} \E \left[\|u\|^3 \|uu^\top - I_d\| \right] \le \frac{L_H \nu (d+6)^\frac 32 }{6}\left[\frac{\sqrt{5}}{2}(d+4)+1\right].
\]
\end{proof}

\begin{remark}
Note that~\eqnok{hessian_approx_error} is obtained only under Assumption~\ref{a:fhess_smooth}. However one could obtain an improved bound on the approximation error, by making the more restrictive assumption of interchangeability of differentiation and expectation as follows: $\| H_{f_\nu} - H_f\| = \| \E[\nabla^2 f(x+ \nu u)] -\nabla^2 f(x)\| \leq \E \| \nabla^2 f(x+ \nu u) -\nabla^2 f(x) \| \leq  L_H \nu \E\| u\| \leq  L_H \nu \sqrt{d} $. While this provides an improved dependency on $d$, we remark that this improvement does not translate to the improvement in the number of zeroth-order oracle calls, at least for the cubic regularized method as discussed in Section~\ref{sec:zocubic}.
\end{remark}

\begin{remark}
Recall from Section~\ref{sec:prelim} that one could obtain high-probability results via the approach proposed in~\cite{GhaLan12,lan2016conditional} under sub-Gaussian tail assumption on the function $F$. To allow for functions $F$, that have heavy-tails, one could also leverage the spectral truncation argument to construct a robust Hessian estimators; see for example,~\cite{minsker2018sub}. Let $\phi \colon  \mathbb{R} \rightarrow \mathbb{R} $ be a non-decreasing function such that
$$
-\log(1  - x + x^2 /2 )\leq \phi(x) \leq  \log(1 + x + x^2 /2), ~~\forall x \in \mathbb{R}.
$$
Recall the definition of a spectral function below.
\begin{definition}
Let $A \in \mathbb{R}^{d \times d}$ be a real symmetric matrix with eigenvalue decomposition $A = U \Lambda U^\top$ where $U$ is the matrix of eigenvectors of $A$ and $\Lambda$ is a diagonal matrix of eigenvalues $(\lambda_1, \ldots, \lambda_d)$.  A real-valued function $\phi(\cdot)$ is a spectral function if it acts on the matrix as follows: $\phi(A) = U \phi(\Lambda) U^\top$ where
\begin{align*}
\phi(\Lambda) = \phi \left(\begin{bmatrix} \lambda_{1} & & \\
    & \ddots & \\
    & & \lambda_{d}\end{bmatrix} \right) = \begin{bmatrix}  \phi(\lambda_{1}) & & \\
    & \ddots & \\
    & & \phi(\lambda_{d})
    \end{bmatrix}
\end{align*}
\end{definition}
Then, we define the robust Hessian estimator as
\begin{align*}
\tilde{H}(\pp, \dd, u) = \frac{1}{\kappa}\cdot \phi\bigl [ \kappa \cdot  H_{\nu} (\pp, \dd, u) \bigr ],
\end{align*}
where $\kappa >0$ is a tuning parameter.   This provides us with a robust Hessian estimator that allows for the function $F$ to have heavy tails. Furthermore, the more standard median-of-means estimator~\cite{nemyud:83} provides a robust gradient estimator as well. A thorough treatment of the estimation error of the robust Hessian and gradient follows from an analysis similar to that of~\cite{minsker2018sub} and~\cite{nemyud:83} respectively, although we do not outline the details in the current paper. We also remark that while the spectral truncation argument makes the estimator robust, the computational advantage of the vanilla estimator in Equation~\ref{Hessian_estimator} is lost.
\end{remark}

\subsection{Zeroth-Order Stochastic Cubic Regularized Newton Method}\label{sec:zocubic}

Our goal in this subsection is to provide a second-order algorithmic framework using the estimated gradient and Hessian based on Stein's identities. In particular, we present a zeroth-order stochastic cubic regularized Newton method in Algorithm~\ref{alg_ZSCRN}. Note that the output of this algorithm, similar to the other algorithms presented in this paper for nonconvex problems, is a random index from the generated trajectory. In order to analyze its complexity, we first state a result due to \cite{NestPoly06-1} that provides optimality conditions of the cubic regularized subproblem in step 2 of Algorithm~\ref{alg_ZSCRN}.

\begin{algorithm} [t]
	\caption{Zeroth-order Stochastic Cubic Regularized Newton Method}
	\label{alg_ZSCRN}
	\begin{algorithmic}
\State Input: $x_0 \in \bbr^d$, smoothing parameter $\nu>0$, non-negative sequence $\alpha_k$, positive integer sequences $m_k$ and $b_k$, iteration limit $N\geq 1$ and probability distribution $P_R(\cdot)$ over $\{1,\ldots,N\}$.

\For {$k = 1, \ldots, N$}
\State 1. Generate $u^{G(H)}_k=[u^{G(H)}_{k,1},\ldots, u^{G(H)}_{k,m_k(b_k)}]$, where $u^{G(H)}_{k,j} \sim N(0,I_d)$, call the stochastic oracle to compute $m_k$ stochastic gradients $G_\nu^{k,j}$ and $b_k$ stochastic Hessians $H_\nu^{k,j}$ according to \eqnok{def_gradk} and \eqnok{Hessian_estimator}, respectively, and take their averages:
\begin{align}
&\bar{G}_\nu^k  =  \frac{1}{m_k} \sum_{j=1}^{m_k} \frac{F(\pp_{k-1}+ \nu u_{k,j}, \dd_{k,j}) - F(\pp_{k-1}, \dd^G_{k,j})}{\nu}~u^G_{k,j},\label{bargrad}\\
&\bar{H}_\nu^k = \frac{1}{b_k} \sum_{i=1}^{b_k} \frac{[F(x_{k-1} + \nu u^H_{k,i}, \dd^H_{k,i})+ F(x_{k-1} - \nu u^H_{k,i}, \dd^H_{k,i})- 2F(x_{k-1},\dd^H_{k,i})]}{2\nu^2}\left(u^H_{k,i} (u^H_{k,i})^\top - I_d\right)\label{barHessian}.
\end{align}
\State 2. Compute
\beq\label{def_xk_cubic}
\pp_k= \underset{x \in \bbr^d}{\argmin} \left\{\tilde f^k(x) \equiv \tilde f(x,x_{k-1},\bar{H}_\nu^k,\bar{G}_\nu^k,\alpha_k)\right\},
\eeq
where
\beq
\tilde f(x,y,H,g,\alpha)= \langle g, x-y \rangle+\frac 12 \langle H(x-y),x-y \rangle+\frac{\alpha}{6}\|x-y\|^3.\label{f_cubic}
\eeq
\EndFor
\State Output: Generate $R$ according to $P_R(\cdot)$ and output $z_R$.
\end{algorithmic}
\end{algorithm}

\vgap
\begin{lemma}[\cite{NestPoly06-1}]\label{opt_condition}
Let $\bar \pp = \underset{x \in \bbr^d}{\argmin} \tilde f(x,y,H,g,\alpha)$. Then, we have
\begin{align*}
&g+H(\bar x -y)+\frac{\alpha}{2}\|\bar x -y\|(\bar x -y) =0,\\
&H+\frac{\alpha}{2}\|\bar x -y\| I_d \succeq 0.
\end{align*}
\end{lemma}
Our next result is the analogous result of Lemma~\ref{bargrad_var} for the averaged Hessian matrices.
\begin{lemma}\label{barhessian_var}
Let $\bar{H}_\nu^k$ be computed by \eqnok{barHessian}, $b_k \ge 4(1+2\log 2d)$. Then under Assumptions~\ref{unbiased_assum} and \ref{smooth_assum}, we have
\begin{align}
\E [\|\bar{H}_\nu^k - \nabla^2 f (x_{k-1})\|^2] &\le \frac{128(1+2\log 2d)(d+16)^4 L^2}{3b_k}  + 3 L_H^2 (d+16)^5 \nu^2,\label{hessian_var2}\\
\E [\|\bar{H}_\nu^k - \nabla^2 f (x_{k-1})\|^3] &\le \frac{160\sqrt{1+2\log 2d}(d+16)^6 L^3}{b_k^\frac 32}+21 L_H^3 (d+16)^\frac{15}{2} \nu^3.\label{hessian_var3}
\end{align}
\end{lemma}

\begin{proof}
First, note that by Theorem 1 in~\cite{tropp2016expected}, we have
\[
\E [\|\bar{H}_\nu^k - \nabla^2 f_\nu (x_{k-1})\|^2] \le \frac{2 C(d)}{b_k^2} \left(\left\|\sum_{i=1}^{b_k}\E[\Delta_{k,i}^2]\right\| +C(d)\E\left[\max_{i}\left\|\Delta_{k,i}\right\|^2\right]\right),
\]
where $\Delta_{k,i} = H_{\nu} (\pp_{k-1}, \dd^H_{k,i}, u^H_{k,i})-\nabla^2 f_\nu (x_{k-1})$ and $C(d)=4(1+2\log 2d)$. Now, noting \eqnok{Hessian_var_bnd}, we have
\[
\E[\|\Delta_{k,i}\|^2] \leq \frac{2(d+16)^4(L_H^2 (d+16)\nu^2+4L^2)}{3},
\]
which together with the above inequality and the fact that
\[
\left\|\sum_{i=1}^{b_k}\E[\Delta_{k,i}^2]\right\|
\le \sum_{i=1}^{b_k}\left\|\E[\Delta_{k,i}^2]\right\|
\le  \sum_{i=1}^{b_k}\E[\|\Delta_{k,i}\|^2],
\]
imply
\beq\label{hessian_var1}
\E [\|\bar{H}_\nu^k - \nabla^2 f_\nu (x_{k-1})\|^2] \le \frac{16(1+2\log 2d)(d+16)^4}{3b_k}\left[4 L^2+L_H^2 (d+16)\nu^2  \right].
\eeq
Combining this inequality with \eqnok{hessian_approx_error}, we obtain \eqnok{hessian_var2}. Moreover, by Holder's inequality we have
\begin{align*}
\E \left[\|\bar{H}_\nu^k - \nabla^2 f_\nu (x_{k-1})\|^3\right]
&\le \E \left[\|\bar{H}_\nu^k - \nabla^2 f_\nu (x_{k-1})\|\cdot \|\bar{H}_\nu^k - \nabla^2 f_\nu (x_{k-1})\|_F^2\right]\\
&\le \left(\E \left[\|\bar{H}_\nu^k - \nabla^2 f_\nu (x_{k-1})\|^2\right] \cdot \E \left[\|\bar{H}_\nu^k - \nabla^2 f_\nu (x_{k-1})\|_F^4\right]\right)^\frac 12.
\end{align*}
Now, by vector-valued Rosenthal's inequality (see, for example, Theorem 5.2 in~\cite{pinelis1994optimum}) and \eqnok{Hessian_var_bnd_fr}, we obtain
\[
\E \left[\|\bar{H}_\nu^k - \nabla^2 f_\nu (x_{k-1})\|_F^4\right] \le \frac{3 \E[\|\Delta_{k,i}\|_F^4]}{b_k^2}
\le
\frac{3(d+16)^8 \left(L_H^4 (d+16)^2 \nu^4+ 9L^4 \right)}{b_k^2},
\]
which together with the above inequality and \eqnok{hessian_var1} imply \eqnok{hessian_var3}.
\end{proof}
\vgap
We now proceed to provide the complexity results for Algorithm~\ref{alg_ZSCRN}. We first require two intermediate results.
\begin{lemma}\label{stationary_lemma1}
Let $\{x_k\}$ be computed by Algorithm~\ref{alg_ZSCRN}. Then under Assumptions~\ref{unbiased_assum} and \ref{smooth_assum}, we have
\begin{align}
&\sqrt{\E[\|x_k-x_{k-1}\|^2]} \nonumber\\
&\ge \max \left\{\sqrt{\frac{\left(\E[\|\nabla f(x_k)\|]-\delta_k^g-\delta_k^H \right)}{L_H+\alpha_k}}
, \frac{-2}{\alpha_k+2L_H} \left[\E[\lambda_{\min} \left(\nabla^2 f(x_k) \right)]+\sqrt{2(\alpha_k+L_H)\delta_k^H}\right] \right\}\label{stationary_condition},
\end{align}
where $\delta_k^g,\delta_k^H>0$ are chosen such that
\beq\label{def_delta_gH}
\E[\|\nabla f(x_{k-1})-\bar{G}_\nu^k\|^2]\le (\delta_k^g)^2, \qquad  \E[\|\nabla^2 f(x_{k-1})-\bar{H}_\nu^k\|^3] \le \left(2(L_H+\alpha_k)\delta_k^H\right)^\frac 32.
\eeq

\end{lemma}

\vgap

\begin{proof}
By the equality condition in Lemma~\ref{opt_condition} and \eqnok{fhess_smooth0}, we have
\begin{align*}
\|\nabla f(x_k)\| &\le \|\nabla f(x_k)-\nabla f(x_{k-1})-\nabla^2 f(x_{k-1})(x_k-x_{k-1})\|+\|\nabla f(x_{k-1})-\bar{G}_\nu^k\|\\
&+\|\nabla^2 f(x_{k-1})-\bar{H}_\nu^k\|\cdot\|x_k-x_{k-1}\|+ \frac{\alpha_k}{2}\|x_k-x_{k-1}\|^2\\
&\le \frac{(L_H+\alpha_k)}{2}\|x_k-x_{k-1}\|^2+ \|\nabla f(x_{k-1})-\bar{G}_\nu^k\|+\|\nabla^2 f(x_{k-1})-\bar{H}_\nu^k\|\cdot \|x_k-x_{k-1}\|\\
&\le (L_H+\alpha_k)\|x_k-x_{k-1}\|^2+ \|\nabla f(x_{k-1})-\bar{G}_\nu^k\|+\frac{\|\nabla^2 f(x_{k-1})-\bar{H}_\nu^k\|^2}{2(L_H+\alpha_k)}.
\end{align*}
Taking expectation from both sides of the above inequality and noting that $\delta_k^g,\delta_k^H$ given in \eqnok{def_delta_gH} are well-defined by properly choosing $m_k$ and $b_k$ in Lemmas~\ref{bargrad_var} and \ref{barhessian_var}, we obtain
\beq\label{stationary_grad}
\frac{\left(\E[\|\nabla f(x_k)\|]-\delta_k^g-\delta_k^H \right)}{L_H+\alpha_k}  \le \E[\|x_k-x_{k-1}\|^2].
\eeq
Also, by smoothness assumption of the Hessian and the inequality relation in Lemma~\ref{opt_condition}
\begin{align*}
\nabla^2 f(x_k) &\succeq \nabla^2 f(x_{k-1})-L_H\|x_k-x_{k-1}\| I_d \succeq \nabla^2 f(x_{k-1})-\bar{H}_\nu^k+\bar{H}_\nu^k-L_H\|x_k-x_{k-1}\| I_d \\
&\succeq \nabla^2 f(x_{k-1})-\bar{H}_\nu^k-\frac{(\alpha_k+2L_H)\|x_k-x_{k-1}\|}{2} I_d,
\end{align*}
which implies that
\[
\frac{(\alpha_k+2L_H)\|x_k-x_{k-1}\|}{2} \ge \lambda_{\min} \left(f(x_{k-1})-\bar{H}_\nu^k \right) - \lambda_{\min} \left(\nabla^2 f(x_k) \right).
\]
Taking expectation from both sides of the above inequality and noting definition of $\delta_k^H$ in \eqnok{def_delta_gH}, we obtain
\[
\sqrt{\E[\|x_k-x_{k-1}\|^2]} \ge \E[\|x_k-x_{k-1}\|] \ge \frac{-2}{\alpha_k+2L_H} \left[\sqrt{2(\alpha_k+L_H)\delta_k^H}+ \E[\lambda_{\min} \left(\nabla^2 f(x_k) \right)]\right].
\]
Combining the above inequality with \eqnok{stationary_grad}, we obtain \eqnok{stationary_condition}.
\end{proof}

\vgap

\begin{lemma}\label{stationary_lemma2}
Let $\{x_k\}$ be computed by Algorithm~\ref{alg_ZSCRN} for a given iteration limit $N \ge 1$. Then under Assumptions~\ref{unbiased_assum} and \ref{smooth_assum}, we have
\beq
\E[\|x_R-x_{R-1}\|^3] \le \frac{36}{\sum_{k=1}^N \alpha_k} \left[f(x_0)-f^*+\sum_{k=1}^N \frac{4 (\delta_k^g)^\frac 32}{\sqrt{3\alpha_k}}+\sum_{k=1}^N \left(\frac{18 \sqrt[4]2}{\alpha_k}\right)^2\left((L_H+\alpha_k)\delta_k^H\right)^\frac 32\right]\label{stationary_upper},
\eeq
where $R$ is an integer random variable whose probability distribution $P_R(\cdot)$ is supported on $\{1,\ldots,N\}$ and given by
\beq\label{def_prob_hesina}
P_R(R=k)= \frac{\alpha_k}{\sum_{k=1}^N \alpha_k} \qquad k=1,\ldots,N,
\eeq
and $\delta_k^g,\delta_k^H>0$ are defined in \eqnok{def_delta_gH}.

\end{lemma}

\begin{proof}
First, note that by \eqnok{fhess_smooth}, \eqnok{def_xk_cubic}, and the fact that $\alpha_k \ge L_H$, we have
\[
f(x_k) \le f(x_{k-1}) + \tilde f^k(x_k)
+\|\nabla f(x_{k-1})-\bar{G}_\nu^k\|\cdot\|x_k-x_{k-1}\|
+\frac 12 \|\nabla^2 f(x_{k-1})-\bar{H}_\nu^k\|\cdot\|x_k-x_{k-1}\|^2.
\]
Moreover, by Lemma~\ref{opt_condition}, we have
\[
\tilde f^k(x_k) = -\frac 12 \langle \bar{H}_\nu^k (x_k-x_{k-1}), (x_k-x_{k-1}) \rangle-\frac{\alpha_k}{3}\|x_k-x_{k-1}\|^3 \le -\frac{\alpha_k}{12}\|x_k-x_{k-1}\|^3.
\]
Combining the above two relations, we obtain
\begin{align*}
\frac{\alpha_k}{12}\|x_k-x_{k-1}\|^3 &\le f(x_{k-1}) -f(x_k) +\|\nabla f(x_{k-1})-\bar{G}_\nu^k\|\cdot\|x_k-x_{k-1}\|\\
&+\frac 12 \|\nabla^2 f(x_{k-1})-\bar{H}_\nu^k\|\cdot\|x_k-x_{k-1}\|^2 \\
&\le f(x_{k-1}) -f(x_k) +\frac{4}{\sqrt{3\alpha_k}}\|\nabla f(x_{k-1})-\bar{G}_\nu^k\|^\frac{3}{2}+ \left(\frac{9 \sqrt2}{\alpha_k}\right)^2 \|\nabla^2 f(x_{k-1})-\bar{H}_\nu^k\|^3 \\
&+ \frac{\alpha_k}{18} \|x_k-x_{k-1}\|^3,\label{hessian_proof1}
\end{align*}
where the last inequality follows from the Young's inequality. Taking expectation from both sides, re-arranging the terms, and noting \eqnok{def_delta_gH}, we obtain
\[
\frac{\alpha_k}{36}\E[\|x_k-x_{k-1}\|^3] \le f(x_0)-f^*+\frac{4 (\delta_k^g)^\frac 32}{\sqrt{3\alpha_k}}+\left(\frac{18 \sqrt[4]2}{\alpha_k}\right)^2\left((L_H+\alpha_k)\delta_k^H\right)^\frac 32.
\]
Summing up the above inequalities, dividing both sides by $\sum_{k=1}^N \alpha_k$, and noting \eqnok{def_prob_hesina}, we obtain \eqnok{stationary_upper}.
\end{proof}

\vgap

\begin{theorem}\label{thm:main_cubic}
Let $\{x_k\}$ be computed by Algorithm~\ref{alg_ZSCRN} for a given iteration limit $N \ge 1$. Moreover, assume that the parameters are set to
\begin{align}
\alpha_k &=L_H, \qquad \nu \le \frac 12 \min\left\{\sqrt{\frac{L_H \epsilon}{36(d+16)^5}},\frac{\epsilon}{L  (d+3)\frac 32} \right\}, \qquad N=\frac{12 \sqrt{L_H}(f(x_0)-f^*)}{\epsilon^\frac 32},\nonumber\\
b_k &=\frac{2L^2}{L_H}\left(4(d+16)^2\right)^4 \left(\frac{\sqrt[3]{1+2\log2d}}{\epsilon}\right), \qquad m_k = \frac{26(d+5)(B^2+\sigma^2)}{\epsilon^2}.\label{def_param_hessian}
\end{align}
Then under Assumptions~\ref{unbiased_assum} and \ref{smooth_assum}, we have
\beq\label{local_min}
5\sqrt{\epsilon} \ge \max \left\{\sqrt{\E[\|\nabla f(x_R)\|]}
, \frac{-5}{8\sqrt{L_H}}\E[\lambda_{\min} \left(\nabla^2 f(x_R) \right)] \right\},
\eeq
where $R$ is uniformly distributed over $\{1,\ldots,N\}$. As a consequence, to obtain an $\epsilon$ second-order stationary point of the problem, the total number of samples required to compute the gradient and Hessian are, respectively, bounded by
\[
{\cal O} \left(\frac{d}{\epsilon^\frac{7}{2}}\right), \qquad \qquad \tilde {\cal O} \left(\frac{d^4}{\epsilon^\frac{5}{2}}\right).
\]
\end{theorem}

\vgap

\begin{proof}
First, note that by \eqnok{def_param_hessian}, Lemmas~\ref{bargrad_var}, and \ref{barhessian_var}, we can ensure that \eqnok{def_delta_gH} is satisfied by $\delta_k^g=2\epsilon/5$ and $\delta_k^H=\epsilon/138$. Moreover, by Lemma~\ref{stationary_lemma2}, we have
\[
\E[\|x_R-x_{R-1}\|^3] \le \frac{1}{L_H^{\frac 32}} \left[\frac{12\sqrt{L_H}(f(x_0)-f^*)}{N}+7\epsilon^\frac 32\right].
\]
Hence, by choosing $N$ according \eqnok{def_param_hessian}, and noting Lemma~\ref{stationary_lemma2}, we obtain \eqnok{local_min}.
Therefore, $x_R$ is an $4\epsilon$ second-order stationary point of the problem. Finally, note that the total number of required samples to obtain such a solution is bounded by
\[
\sum_{k=1}^N m_k = {\cal O} \left(\frac{d}{\epsilon^\frac{7}{2}}\right), \qquad \sum_{k=1}^N b_k = \tilde {\cal O} \left(\frac{d^4}{\epsilon^\frac{5}{2}}\right)
\]

\end{proof}

\begin{remark}
Note that~\cite{tripuraneni2017stochastic} provide a high-probability complexity result for stochastic Newton method with inexact gradient and Hessian information of the order $\tilde {\cal O} \left(\epsilon^{-3.5}\right)$. This dependence on $\epsilon$ is better compared to algorithms that only use stochastic first-order information to avoid saddle points. They mainly focus on sub-sampled Newton method common in the finite-sum setting and require their stochastic Hessians to be almost-surely bounded. However, this assumption does not imply the zeroth-order Hessian estimators in Equation~\ref{Hessian_estimator} are bounded almost-surely, which complicates the analysis.
\end{remark}

\begin{remark}
Note that by Theorem~\ref{thm:main_cubic}, the total number of calls to the zeroth-order oracle is of the order $\tilde{\cal O} \left(\frac{d}{\epsilon^{3.5}}\right)$ when $\epsilon \approx d^{-3}$. This shows the advantage of using the (estimated) second-order information for converging to high-accuracy second-order stationary points. The linear dependence on $d$ is the price to pay for having access to only zeroth-order information, similar to the previous sections. Furthermore, depending on the quality of the solution required, a wide variety of intermediate complexity results are possible, thereby providing practical flexibility.
\end{remark}

\section{Discussion}\label{sec:theend}
In this work, we propose and analyze zeroth-order stochastic approximation algorithms for convex and nonconvex problems motivated by modern machine learning challenges. Specifically, we provide zeroth-order algorithms to deal with constraints, dimensionality and saddle-points in nonconvex stochastic optimization problems. While our focus was on general stochastic optimization problems, one could naturally obtain better rates in the case of finite-sum optimization problems with various variance reduction techniques. Several concrete extensions are possible for future work. The performance of conditional gradient algorithm in the high-dimensional constrained optimization setting is not well-explored; the interaction between the geometry of the constraint set, sparsity structure and zeroth-order information is extremely interesting to explore. Obtaining regret bounds for the non-convex problems considered in this work is more challenging. Furthermore, lower bounds can be explored for the cases considered in this paper when $f$ is nonconvex. Finally, obtaining second-order stationarity results in the constrained setting is more challenging. We plan to extend our results for these setting in the future.
\bibliographystyle{alpha}
\bibliography{optzero}

\end{document}